\documentclass[preprint,11pt]{elsarticle_pre}
\usepackage{hyperref}

\usepackage{float}
\usepackage{amssymb}
\usepackage[dvipsnames]{xcolor}

\usepackage[english]{babel} 

\usepackage{amsmath,amssymb,latexsym,amsthm,enumerate,epsfig,graphicx}
\emergencystretch=\maxdimen
\newtheorem{definition}{Definition}[section]
\newtheorem{theorem}{Theorem}[section]
\newtheorem{proposition}{Proposition}[section]
\newtheorem{lemma}{Lemma}[section]

\newtheorem{remark}{Remark}[section]
\numberwithin{equation}{section}

\newif\ifcomment \commentfalse

\newcommand{\remove}[1]{}

\newenvironment{vardesc]}[1]{%
\settowidth{\parindent}{#1: \ }
\makebox{#1:}}{}

\begin{document} 

\begin{frontmatter}

\title{Local existence of smooth solutions \\ to multiphase models in two space dimensions}

\author[Bianchini]{Roberta Bianchini}


\ead{bianchin@mat.uniroma2.it}
\fntext[Bianchini]{Dipartimento di Matematica, Universit\`a degli Studi di Roma "Tor Vergata", via della Ricerca Scientifica 1, I-00133 Rome, Italy - Istituto per le Applicazioni del Calcolo "M. Picone", Consiglio Nazionale delle Ricerche, via dei Taurini 19, I-00185 Rome, Italy. }

\author[Natalini]{Roberto Natalini}
\ead{roberto.natalini@cnr.it}
\fntext[Natalini]{Istituto per le Applicazioni del Calcolo "M. Picone", Consiglio Nazionale delle Ricerche, via dei Taurini 19, I-00185 Rome, Italy.}
\begin{abstract}
In this paper, we consider a class of models for multiphase fluids, in the framework of mixture theory. The considered system, in its more general form, contains both the gradient of a hydrostatic pressure, generated by an incompressibility constraint,  and the gradient of a compressible pressure depending on the volume fractions of some of the different phases. To approach these systems, we define an approximation based on the \emph{Leray} projection, which involves the use of the \emph{Lax} symbolic symmetrizer for hyperbolic systems and paradifferential techniques. In two space dimensions, we prove its well-posedness and convergence to the unique classical solution to the original system. In the last part, we shortly discuss the difficulties in the three dimensional case.
\end{abstract}


\end{frontmatter}

\section{Introduction}
We consider a fluid composed of two different phases, with respective volume fractions $B, L \in [0,1],$ depending on time $t \in [0, T]$, for some $T>0$, and space $x \in \mathbb{R}^{d},$ and satisfying the following system:

\begin{equation}
\label{BL_system1}
\begin{cases}
& \partial_{t}B + \nabla \cdot (B v_{S}) = \Gamma_{B}, \\
& \partial_{t}L + \nabla \cdot (L v_{L}) = \Gamma_{L}, \\
& \partial_{t}(B{v}_{S})+\nabla \cdot (B v_{S} \otimes v_{S}) + \gamma \nabla B + B \nabla P = (M-\Gamma_{L})v_{L}-Mv_{S}, \\
& \partial_{t}(Lv_{L})+\nabla \cdot (L v_{L} \otimes v_{L}) + L \nabla P = -(M-\Gamma_{L})v_{L}+Mv_{S}, \\
& B+L=1, \\
& \Gamma_{B}+\Gamma_{L}=0, \\
\end{cases}
\end{equation}

where the $d$-dimensional vectors $v_{S}, v_{L}$ are the velocities of $B$ and $L$ respectively, $\nabla P$ is the gradient of an incompressible pressure term, $\gamma, M$ are experimental constant values and $\Gamma_{B}, \Gamma_{L}$ are given source terms, possibly dependent on $B$ and $L$. In (\ref{BL_system1}), the momentum equations for the solid and the liquid phases are different: in the first one, $\gamma \nabla B$ is the excess stress tensor, while, following \cite{astanin}, \cite{Farina}, there is no excess stress tensor for the liquid part. Summing the first and the second equation, using the two last conditions in (\ref{BL_system1}), and setting $L=1-B$, yields:
\begin{equation}
\label{BL_System}
\begin{cases}
& \partial_{t}B + \nabla \cdot (B v_{S}) = \Gamma_{B}, \\
& \partial_{t}{v}_{S}+{v}_{S} \cdot \nabla {v}_{S} +\frac{\gamma \nabla B}{B}+ \nabla {P}=\frac{(M+\Gamma_{B})(v_{L}-v_{S})}{B}, \\
& \partial_{t}{v}_{L}+{v}_{L} \cdot \nabla {v}_{L} + \nabla {P}=\frac{M(v_{S}-v_{L})}{(1-B)}, \\
& \nabla \cdot (Bv_{S} + (1-B)v_{L}) = 0, \\
\end{cases}
\end{equation}

where the last equation represents the incompressibility of the mixture as a whole.
More generally, system (\ref{BL_System}) is the two phases case of a wide class of multiphase models of fluids of an arbitrary number $N$ of constituents. These models arise from mixture theory, see for instance \cite{Rajagopal}, \cite{Bowen1}, \cite{Bowen2}, \cite{Muller1}, \cite{Muller2}, \cite{Ehlers}, and their general formulation is as follows:

\begin{equation}
\label{general_mixture_model}
\begin{cases}
& \rho_{n}(\partial_{t}\phi_{n}+\nabla \cdot (\phi_{n} v_{n}))=\Gamma_{n}, \\
& \rho_{n}(\partial_{t}(\phi_{n}v_{n})+\nabla \cdot (\phi_{n} v_{n} \otimes v_{n})) = \nabla \cdot \tilde{T}_{n} + m_{n} + \Gamma_{n} v_{n}, \\
& \sum_{n=1}^{N}\phi_{n}=1, \\
& \sum_{n=1}^{N} \Gamma_{n}=0, \\
\end{cases}
\end{equation}

where 

\begin{itemize}
\item $\rho_{n}$ is the density of the $n^{th}$ phase, assumed to be constant, in accordance with the incompressibility condition,
\item $\phi_{n}$ is the volume fraction of each constituent,
\item $v_{n}$ is the specific velocity,
\item $\Gamma_{n}$ is the mass exchange rate between different phases,
\item $\tilde{T}^{n}$ is the partial stress tensor,
\item $m_{n}$ is the interaction force, which is related to interactions between different phases across the interface.
\end{itemize}

For more details about the derivation of model (\ref{general_mixture_model}) see, for instance, \cite{Farina}. These models arises in many fields, as tumor growth and vasculogenesis in \cite{astanin}, biological tissues and porous media in \cite{Farina}. In particular, among  several applications, we refer to the model proposed in \cite{cdnr}, which describes biological structures called \emph{biofilms}, composed by four different phases: bacteria, dead bacteria, extracellular polymeric matrix, and the liquid phase, with respective volume fractions $B, D, E, L.$ This model satisfies the following equations:

\begin{equation}
\label{BDEL}
\begin{cases}
& \partial_{t}B+\nabla \cdot (Bv_{S})=\Gamma_{B}, \\
& \partial_{t}D+\nabla \cdot (Dv_{S})=\Gamma_{D}, \\
& \partial_{t}E+\nabla \cdot (Ev_{S})=\Gamma_{E}, \\
& \partial_{t}L+\nabla \cdot (Lv_{L})=\Gamma_{L}, \\
& \partial_{t}((1-L)v_{S}) + \nabla \cdot ((1-L)v_{S} \otimes v_{S}) + \gamma \nabla (1-L) + (1-L) \nabla P \\
& =(M-\Gamma_{L})v_{L}-Mv_{S}, \\
& \partial_{t}(Lv_{L}) + \nabla \cdot (Lv_{L} \otimes v_{L}) + L \nabla P = -(M-\Gamma_{L})v_{L}+Mv_{S}, \\
& B+D+E+L=1, \\
& \Gamma_{B}+\Gamma_{D}+\Gamma_{E}+\Gamma_{L}=0, \\
\end{cases}
\end{equation}

where $v_{\phi}, \Gamma_{\phi}$ are respectively the velocities and the source terms for $\phi=B, D, E, L,$ and $\nabla P$ is the incompressible pressure. Actually, system (\ref{BL_System}) is just the two phases case of (\ref{BDEL}), where the three "solid" species are lumped together. Besides, system (\ref{BL_System}) is worth considering, since many models for biological applications are constituted by two phases as, for instance, the one presented in \cite{Farina}. In the following, we will focus on system (\ref{BL_System}) for simplicity reasons, but the arguments will be briefly generalized to the  case of system (\ref{BDEL}) in Section 4.

Although mixture models are largely diffused, up to now the analytical theory has been mainly developed in one space dimension, see for instance \cite{Parabello},  \cite{Stewart}, \cite{Ystrom}, and \cite{Ishii}, while some results about linear stability and numerical approximations were considered in \cite{Gud1}. More recently, a complete analytical study of the one dimensional biofilms model (\ref{BDEL}) and the related two phases system (\ref{BL_System}), with the proof of the global existence and uniqueness of smooth solutions and the analysis of their asymptotic behavior  for initial data, which are small perturbations of the equilibrium point, were given in \cite{Bianchini1}. Actually, in the one dimensional case, a great simplification occurs due to the fact that, by adding the mass balance equations of (\ref{general_mixture_model}) for $n=1, \cdots, N,$ we get the incompressibility condition for the mixture as a whole, which, in the one dimensional case, reads
\begin{equation}
\label{incompressibility_condition}
\sum_{n=1}^{N} \partial_{x}(\phi_{n} v_{n}) = 0.
\end{equation}
After some calculations, the equality (\ref{incompressibility_condition}) allows us to solve for the incompressible pressure $\nabla P$ in (\ref{general_mixture_model}). Besides, the remaining system is symmetrizable hyperbolic (see \cite{Majda}, \cite{Benzoni} \cite{Taylor}), and so the standard theory applies. On the other hand, in several space dimensions there is not a simple way to deal with the term $\nabla P$. More precisely, considering model (\ref{BL_System}), the incompressibility condition is the following:
\begin{equation}
\label{incompressibility_BDEL}
\nabla \cdot (Bv_{S}+(1-B)v_{L})=0.
\end{equation} 
In order to work using a divergence free formulation, we define

\begin{equation}
\label{change1}
w:=Bv_{S}+(1-B)v_{L}
\end{equation}
and we could try to apply classical methods used for incompressible fluids, see \cite{Temam}, \cite{Majda}, \cite{Valli}, and \cite{Bertozzi}, which are essentially based on the projection of the velocity field onto the space of the divergence free vectors. 

However, in our case, even in the divergence free variables, there are some difficulties. The first one is given by the interaction between the \emph{Friedrichs} symmetrizer of the hyperbolic part of system (\ref{BL_System}) and the gradient of the incompressible pressure term. Actually, the scalar product induced by the classical symmetrizer does not preserve the orthogonality of the gradient of the incompressible pressure with respect to the divergence free average velocity. This happens since the symmetrizer and the pressure part of the system do not commute and, moreover, their commutator is still a first order operator, see Section 2 below. Therefore, we cannot get rid of the incompressible pressure, unlike in the case of the incompressible Euler equations, see for instance \cite{Lions}. Furthermore, it is not obvious how to get useful energy estimates in Sobolev spaces for system (\ref{BL_System}), since our hydrostatic pressure does not possess enough regularity in space. In fact, looking at the elliptic equation for the pressure $P$, which is obtained applying the divergence operator to the momentum equations in system (\ref{BL_System}), we have
\begin{equation}
\label{pressure_elliptic_equation1}
\Delta P= -\sum_{j=1}^{d}\sum_{i=1}^{d}  \partial_{x_{j}}w_{i} \partial_{x_{i}}w_{j} - \nabla \cdot \nabla \cdot (B(1-B)z \otimes z) - \gamma \Delta B,
\end{equation}
where $z:=v_{S}-v_{L}.$
Let us compare (\ref{pressure_elliptic_equation1}) with the elliptic equation for the pressure $P^{E}$ of the incompressible Euler equations with velocity $v^{E}$, see \cite{Lions}, namely 
\begin{equation}
\label{pressure_Euler}
\Delta P^{E} =- \sum_{j=1}^{d}\sum_{i=1}^{d}  \partial_{x_{j}}{v^{E}}_{i} \partial_{x_{i}}{v^{E}}_{j}.
\end{equation}
Starting from velocity fields $w,z$ in (\ref{pressure_elliptic_equation1}) and $v^{E}$ in (\ref{pressure_Euler}) with the same $H^{s}$ regularity for some $s>[d/2]+1,$ our pressure $P$ in (\ref{pressure_elliptic_equation1}) is only in  $H^{s}$, while $P^{E}$ in (\ref{pressure_Euler}) is in $H^{s+1}$. So, because of this lack of regularity, which is due not only to the inertial term
$\nabla \cdot (B(1-B) z \otimes z),$ but also to the compressible pressure term $\gamma \nabla B$, we are unable to close the energy estimates for system (\ref{BL_System}).

For all these reasons, the different approaches used for incompressible fluids do not work for (\ref{BL_System}). For instance, even if the numerical simulations in \cite{cdnr}, which use the \emph{Chorin-Temam} projection method  \cite{Temam}, seem to yield some reliable results,   we do not know how to prove any rigorous convergence result for  this approximation scheme in this case. In fact, while the $L^{2}$-norm of the projected solution is estimated step by step by the $L^{2}$-norm of the non-projected vector, thanks to the \emph{Hodge} decomposition theorem \cite{Temam}, this property no more holds for the scalar product induced by the symmetrizer and so we are unable to control the energy estimates. This structural difficulty is also the reason why the singular perturbation approximation in \cite{Bianchini}, which can be viewed as a continuous version of the projection method, does not work for system (\ref{BL_System}). Also, we are not able to prove the convergence of the approximation used by \emph{Valli} \& \emph{ Zajazckowski} in \cite{Valli} to solve the incompressible Euler equations, since, again, we cannot get the necessary energy estimates from the elliptic equation (\ref{pressure_elliptic_equation1}). For completeness, we notice that the same holds for the artificial compressibility method of \emph{Temam} in \cite{Temam}, since there is no classical symmetrizer for the related approximating compressible system and the \emph{Lax} symmetrizer that we have found does not satisfy the assumptions required in studying singular perturbations approximations, as in \cite{Grenier}.

In spite of these negative remarks, in this paper we prove the convergence of one approximation to system (\ref{BL_System}), made by the composition of some smoothing operators and the \emph{Leray} projector, see \cite{Bertozzi} and \cite{Bianchini} for different applications of this technique.
Here, the main idea is as follows. First, we apply the projector onto the space of the vectors such that the averaged velocity $w$ is divergence free. Then, we consider the highest order part of the paradifferential operator associated to the projected system (\ref{BL_System}), which has a useful structural property. It can be verified that the highest order part is a strongly hyperbolic operator of the first order, therefore it is possible to find a \emph{Lax} symmetrizer for it. The construction of this symmetrizer is essentially based on the techniques developed in \cite{Metivier}, which are combined here to some ideas in \cite{Grenier}. We point out that, the main point here is to symmetrize the whole projected operator, rather than just use the symmetrizer of the hyperbolic part of (\ref{BL_System}). Using paradifferential calculus, we are able to establish some uniform energy estimates and the convergence of this method to (\ref{BL_System}), as well as in the case of the more general model (\ref{BDEL}),  both in two space dimensions. 

\subsection{Plan of the paper}
The paper is organized as follow. In Section 2 we discuss the general setting and the main properties of the two phases system (\ref{BL_System}) in two space dimensions, by emphasizing the difficulties for the various  formulations of the problem. Section 3 is devoted to the definition and well-posedness of our approximation, using an approach based on paradifferential calculus, and a proof of its convergence. In Section 4, we show how to apply the previous arguments to the more general system (\ref{BDEL}), always in two space dimensions. Finally, in Section 5 we discuss the difficulties we have found to extend these results to the the three dimensional case.  

\section{General setting}
\subsection{Basic formulation}
Let $\textbf{u}=(B, v_{S}, v_{L})$. System (\ref{BL_System}) can be written in the following compact form:

\begin{equation}
\label{BL_System_Compact1}
\begin{cases}
& \partial_{t}\textbf{u}+\sum_{j=1}^{d}A_{j}(\textbf{u})\partial_{x_{j}}\textbf{u}+F_{P}=G(\textbf{u}), \\
& \nabla \cdot (Bv_{S} + (1-B)v_{L})=0, \\
\end{cases}
\end{equation}

where the term $F_{P}$ is given by the gradient of the hydrostatic incompressible pressure

\begin{equation}
\label{Fp}
F_{P}=(0, \nabla P, \nabla P),
\end{equation}

and the source term have the following expression

\begin{equation}
\label{Source}
G(\textbf{u})=(\Gamma_{B}, \Gamma_{v_{S}}, \Gamma_{v_{L}}),
\end{equation}

where

\begin{equation}
\label{Gamma}
\begin{array}{ccc}
\Gamma_{B}=B(k_{B}(1-B)-k_{D}), & \Gamma_{v_{S}}=\frac{(M+\Gamma_{B})(v_{L}-v_{S})}{B}, & \Gamma_{v_{L}}=\frac{M(v_{S}-v_{L})}{(1-B)},
\end{array}
\end{equation}

and $k_{B}, k_{D}, M$ are experimental constants. The initial data related to (\ref{BL_System_Compact1}) are the following:
\begin{equation}
\label{initial_data}
\textbf{u}(0,x)=\textbf{u}_{0}(x)=(B_{0}(x), v_{{S}_{0}}(x), v_{{L}_{0}}(x)) ~~~ \text{such that} ~~~ \nabla \cdot (B_{0} v_{{S}_{0}} + (1-B_{0}) v_{{L}_{0}})=0.
\end{equation}
 
Although most of the calculations in this first section hold in the general $d$-dimensional case, we limit our consideration only to the two dimensional case. In one space dimension, in fact, system (\ref{BL_System_Compact1}) is a particular version of that already discussed in \cite{Bianchini1}, while in three space dimensions there are some structural problems that lead to technical difficulties, as we will see in Section 5. Setting $d=2$, system (\ref{BL_System_Compact1}) reads

\begin{equation}
\label{BL_System_Compact}
\begin{cases}
& \partial_{t}\textbf{u}+A_{1}(\textbf{u})\partial_{x}\textbf{u}+A_{2}(\textbf{u})\partial_{y}\textbf{u}+F_{P}=G(\textbf{u}), \\
& \nabla \cdot (Bv_{S} + (1-B)v_{L})=0, \\
\end{cases}
\end{equation}

with $F_{p}$ in \eqref{Fp}, $G(\textbf{u})$ in \eqref{Source} and the initial data $\textbf{u}_{0}$ in \eqref{initial_data}. The flux matrices are:
\begin{equation}
\label{A12}
A_{1}(\textbf{u})=\left( \begin{array}{ccccc}
v_{S_{1}} & B & 0 & 0 & 0 \\
\frac{\gamma}{B} & v_{S_{1}} & 0 & 0 & 0 \\
0 & 0 & v_{S_{1}} & 0 & 0 \\
0 & 0 & 0 &  v_{L_{1}}  & 0 \\
0 & 0 & 0 & 0 &  v_{L_{1}}\\
\end{array} \right), ~~~~~ A_{2}(\textbf{u})=\left( \begin{array}{ccccc}
v_{S_{2}} & 0 & B & 0 & 0 \\
0 & v_{S_{2}} & 0 & 0 & 0 \\
\frac{\gamma}{B}  & 0 & v_{S_{2}} & 0 & 0 \\
0 & 0 & 0 &  v_{L_{2}}  & 0 \\
0 & 0 & 0 & 0 &  v_{L_{2}}\\
\end{array} \right).
\end{equation}

\textbf{Assumption.}
\label{assumption}
From (\ref{A12}), the terms $B$ and $(1-B)$ cannot vanish. Then, we take $0 < B < 1.$ 
\begin{remark}
\label{translation}
The assumption above is quite natural, in fact, from the mass balance equation for $B$ in (\ref{BL_System}), if the initial data $B_{0}$ (and $1-B_{0}$) in (\ref{initial_data}) does not vanish for all $x \in \mathbb{R}^{d}$, then, under some standard assumptions of regularity, $B(t,x)$ (and $1-B(t,x)$) cannot vanish too. In particular, if   $B_{0} \in W^{1,\infty}(\mathbb{R}^{2})$ and  $v_{S} \in L^{1}([0,T], Lip(\mathbb{R}^{2}))$, the strict positivity of $B$ (and $1-B$) follows by the results in \cite{Lad}. 

In the following, we prove that, fixing a constant value $\bar{B}$ and taking $B_{0}$ such that $B_{0}-\bar{B} \in H^{s}(\mathbb{R}^{2})$, with $s>[d/2]+1=2$, then $(B-\bar{B}, v_{S}) \in C([0,T], H^{s}(\mathbb{R}^{2})) \cap C^{1}([0,T], H^{s-1}(\mathbb{R}^{2}))$, in accordance with the assumptions of Proposition 1 in \cite{Lad}.
\end{remark}

As discussed in Remark \ref{translation}, system (\ref{Biofilms_Compact_Projector}) is singular in $B=0,$ then the unknown $B$ cannot belong to $L^{2}(\mathbb{R}^{2}).$ In order to work in the natural setting of the Sobolev spaces, we make a slight modification. From the form of the source term ${G}$ in (\ref{Source})--\eqref{Gamma}, the admissible equilibrium point of system (\ref{BL_System}) is the following:
\begin{equation}
\label{equilibrium_point}
\bar{\textbf{u}}=(\bar{B}, \bar{v}_{S}, \bar{v}_{L})=(1-\frac{k_{D}}{k_{B}}, \bar{v}, \bar{v}),
\end{equation}
where $\bar{v}$ is a two dimensional constant vector arbitrarily chosen. Taking $\bar{v}=0,$ we have
\begin{equation}
\label{equilibrium_point1}
\bar{\textbf{u}}=(\bar{B}, \textbf{0}, \textbf{0}). 
\end{equation}
In this section, to simplify the presentation, we define the translated system, which will be considered in Section 3.
Let
\begin{equation}
\label{new_variable}
\tilde{\textbf{u}}=(\tilde{B}, \tilde{v}_{S}, \tilde{v}_{L}):=\textbf{u}-\bar{\textbf{u}},
\end{equation}
with $\bar{\textbf{u}}$ in (\ref{equilibrium_point1}). Then, we will study the following system:
\begin{equation}
\label{BL_compact_translated1}
\begin{cases}
& \partial_{t}\tilde{\textbf{u}}+\sum_{j=1}^{2}{A}_{j}(\tilde{\textbf{u}}+\bar{\textbf{u}})\partial_{x_{j}}\tilde{\textbf{u}}+{F}_{P}={G}(\tilde{\textbf{u}}+\bar{\textbf{u}}), \\
& \nabla \cdot ((\tilde{B}+\bar{B})\tilde{v}_{S}+(1-(\tilde{B}+\bar{B}))\tilde{v}_{L})=0, \\
\end{cases}
\end{equation}
with initial data 
\begin{equation}
\label{translated_initial}
\tilde{\textbf{u}}(0,x)=\tilde{\textbf{u}}_{0}={\textbf{u}}_{0}-\bar{\textbf{u}},
\end{equation}
and ${\textbf{u}}_{0}$ in (\ref{initial_data}).
We provide now the definition of classical local solutions to (\ref{BL_System}).
\begin{definition}
\label{classical_solutions}
Let $s>2$ be fixed. The function $\tilde{\textbf{u}}=(\tilde{B}, \tilde{v}_{S}, \tilde{v}_{L})$ is a classical solution to system (\ref{BL_System_Compact1}), if $\tilde{\textbf{u}} \in C([0,T], H^{s}(\mathbb{R}^{2}))\cap C^{1}([0,T], H^{s-1}(\mathbb{R}^{2}))$ for any time $T>0,$ and $\tilde{\textbf{u}}$ solves  system (\ref{BL_compact_translated1}) in the classical sense, with initial data $\tilde{\textbf{u}}_{0} \in H^{s}(\mathbb{R}^{2})$ in (\ref{translated_initial}, where $P$ is a function such that $\nabla P \in C([0,T], H^{s-1}(\mathbb{R}^{2})).$
\end{definition}

In the remainder of this section, we just omit the tilde to simplify the notations. Considering (\ref{A12}), it is easy to find a diagonal matrix that symmetrizes the first order part $A_{1}(\textbf{u})\partial_{x}\textbf{u}, A_{2}(\textbf{u})\partial_{y}\textbf{u}$ of system (\ref{BL_System_Compact}. The \emph{Friedrichs} (or classical) symmetrizer is
\begin{equation}
\label{A0}
S_{0}(\textbf{u})=diag(\gamma/B, B, B, (1-B), (1-B)).
\end{equation}
The existence of this symmetrizer for $A_{1}(\textbf{u}), A_{2}(\textbf{u})$ implies that, disregarding the pressure term, system (\ref{BL_System_Compact}) is hyperbolic. Nevertheless, that classical symmetrizer is not useful to close some energy estimates, since we have to deal also with the incompressible pressure term $F_{P}$. In fact, in the Sobolev spaces $H^{s}(\mathbb{R}^{2})$ with $s>2$, when we take the $s$-derivative of system (\ref{BL_System_Compact}) and multiply by $\nabla^{s}\textbf{u}$ in order to get energy estimates, the right-hand side of the equation contains the following scalar product in $L^{2}(\mathbb{R}^{2})$:
$$(S_{0}(\textbf{u})\nabla^{s}F_{P}, \nabla^{s}\textbf{u})_{0}=((0, B\nabla^{s+1}P,  (1-B)\nabla^{s+1}P),(\partial_{x}^{s}B, \nabla^{s}v_{S}, \nabla^{s}v_{L}))_{0}$$
$$=(B\nabla^{s+1}P , \nabla^{s}v_{S})_{0}+((1-B)\nabla^{s+1}P,  \nabla^{s}v_{L})_{0}.$$
Unfortunately, taking $\textbf{u} \in H^{s}(\mathbb{R}^{2})$, the pressure term $P$ has not enough regularity, as shown by the elliptic equation (\ref{pressure_elliptic_equation1}), and then we are unable to close our estimates.
Besides, the symmetrizer (\ref{A0}) depends on the variable $\textbf{u},$ whose components do not depend explicitly on the average velocity $Bv_{S}+(1-B)v_{L},$ which is, instead, the divergence free vector field associated to (\ref{BL_System_Compact}). Then, we can try to introduce the new variables
\begin{equation}
\label{Change_variables}
\begin{array}{cc}
w:=Bv_{S}+(1-B)v_{L}, &  z:=v_{S}-v_{L}, \\
\end{array}
\end{equation}
so setting $\textbf{v}=(B, w, z).$ It can be easily seen that, passing to the new variable $\textbf{v},$ the fourth equation of (\ref{BL_System}) yields $\nabla \cdot {w}=0,$ which is exactly the incompressibility condition for the mixture as a whole. Moreover, the equation for the average velocity $w,$ which is
\begin{equation}
\label{average_velocity_equation}
\partial_{t}w+w \cdot \nabla w + \nabla \cdot (B(1-B)z\otimes z) + \gamma \nabla B + \nabla P=0,
\end{equation}
contains the gradient of the incompressible pressure $\nabla P$ alone, without multiplication by any phase volume fraction, while the equation for the relative velocity $z,$
\begin{equation}
\label{relative_velocity_equation}
\partial_{t}z+w \cdot \nabla z + z \cdot \nabla w + z \cdot \nabla ((1-B)z) -B z \cdot \nabla z + \frac{\gamma \nabla B}{B}=-\frac{z(M+\Gamma_{B}(1-B))}{B(1-B)},
\end{equation}
is free from the incompressible pressure. Considering (\ref{Change_variables}), let $\phi(\textbf{u})$ be the diffeomorphism so defined
\begin{equation}
\label{diffeomorphism}
\textbf{v}=(B, w, z)=\phi(\textbf{u})=( B, Bv_{S}+(1-B)v_{L}, v_{S}-v_{L}).
\end{equation}
System (\ref{BL_System_Compact}), can be written in the following compact form:
\begin{equation}
\label{BL_System_Compact_NewVariables}
\begin{cases}
& \partial_{t}\textbf{v}+\sum_{j=1}^{2}\tilde{A}_{j}(\textbf{v})\partial_{x_{j}}\textbf{v}+\tilde{F}_{P}=\tilde{G}(\textbf{v});\\
& \nabla \cdot w=0, \\
\end{cases}
\end{equation}

with initial data 

\begin{equation}
\label{initial_data_new_variables}
\textbf{v}(0,x)=\textbf{v}_{0}(x)=(B_{0}(x), w_{0}(x), z_{0}(x)) ~~~ \text{such that} ~~~ \nabla \cdot w_{0}=0,
\end{equation}

where 

\begin{equation}
\label{Aj_new_variables}
\tilde{A}_{j}(\textbf{v})=(\phi'A_{j}\phi'^{-1})(\phi^{-1}(\textbf{v})), ~~~ \text{for} ~~~ j=1, 2,  ~~~~~ \tilde{G}(\textbf{v})=(\phi'G\phi'^{-1})(\phi^{-1}(\textbf{v})),
\end{equation}
and
\begin{equation}
\label{Pressure_new_variables}
\tilde{F}_{P}=(0, \nabla P,  \textbf{0}).
\end{equation}
 Explicitly,
\begin{equation}
\label{A1_new_variables}
\tilde{A}_{1}(\textbf{v})=\left( \begin{array}{ccccc}
w_{1}+z_{1}(1-2B) & B & 0 & B(1-B) & 0 \\
\gamma+z_{1}^{2}(1-2B) & w_{1}+Bz_{1} & 0 & 2Bz_{1}(1-B) & 0 \\
z_{1}z_{2}(1-2B) & Bz_{2} & w_{1} & Bz_{2}(1-B) & Bz_{1}(1-B) \\
\frac{\gamma}{B}-z_{1}^{2} & z_{1} & 0 &  w_{1}+z_{1}(1-2B) & 0 \\
-z_{1}z_{2} & 0 & z_{1} & 0 &  w_{1}+z_{1}(1-2B)\\
\end{array} \right),  
\end{equation}

\begin{equation}
\label{A2_new_variables}
\tilde{A}_{2}(\textbf{v})=\left(\begin{array}{ccccc}
w_{2}+z_{2}(1-2B) & 0 & B & 0 & B(1-B) \\
z_{1}z_{2}(1-2B) & w_{2} & Bz_{1} & Bz_{2}(1-B) & Bz_{1}(1-B) \\
\gamma+z_{2}^{2}(1-2B) & 0 & w_{2}+Bz_{2} & 0 & 2Bz_{2}(1-B) \\
-z_{1}z_{2} & z_{2} & 0 &  w_{2}+z_{2}(1-2B) & 0 \\
\frac{\gamma}{B}-z_{2}^{2} & 0 & z_{2} & 0 & w_{2}+z_{2}(1-2B)\\
\end{array} \right),
\end{equation}

while

\begin{equation}
\label{G_new_variables}
\tilde{G}(\textbf{v})=(\Gamma_{B}, \textbf{0}, \frac{-z(M+\Gamma_{B}(1-B))}{B(1-B)}),
\end{equation}
with $\Gamma_{B}$ in (\ref{Gamma}). System (\ref{BL_System_Compact_NewVariables}) can be written as
\begin{equation}
\label{paradiff_withoout_proj}
\partial_{t}\textbf{v}+T_{i\tilde{A}(\xi, \textbf{v})}\textbf{v}=T_{\tilde{G}(\xi, \textbf{v})}+[T_{i\tilde{A}(\xi, \textbf{v})}-\sum_{j=1}^{2}\tilde{A}_{j}(\textbf{v})\partial_{x_{j}}]\textbf{v}+[\tilde{G}(\textbf{v})-T_{\tilde{G}(\textbf{v})}],
\end{equation}
where, from \cite{Metivier}, 
\begin{equation}
\label{Paradiff_total}
T_{i\tilde{A}(\xi, \textbf{v})}=\sum_{j=1}^{2}T_{\tilde{A}_{j}(\textbf{v})}\partial_{x_{j}}\textbf{v}
\end{equation}
is the paradifferential operator associated to the $x$-dependent matrix symbol
\begin{equation}
\label{A_symbol}
i\tilde{A}(\xi, \textbf{v})=\sum_{j=1}^{2}i\xi_{j}\tilde{A}_{j}(\textbf{v})=\sum_{j=1}^{2}i\xi_{j}\tilde{A}_{j}(\textbf{v}(t,x)),
\end{equation}
and similarly for $\tilde{G}(\textbf{v})$ and $T_{\tilde{G}(\textbf{v})}.$ The symbolic matrix $\tilde{A}(\xi, \textbf{v})$ in (\ref{A_symbol}) has the following eigenvalues:
\begin{equation}
\label{eigenvalues_A}
\begin{cases}
& \lambda_{1}=\lambda_{2}=(w-Bz)\cdot \xi, \\
& \lambda_{3}=(w+(1-B)z)\cdot \xi, \\
& \lambda_{4}=(w+(1-B)z)\cdot \xi - \sqrt{\gamma}|\xi|, \\
& \lambda_{5}=(w+(1-B)z)\cdot \xi + \sqrt{\gamma}|\xi|. \\
\end{cases}
\end{equation}
They are real, then, as long as we neglect the incompressible pressure $\tilde{F}_{P},$ the remaining system in (\ref{BL_System_Compact_NewVariables}) is hyperbolic, as the symmetrizable system (\ref{BL_System_Compact}) in the old variable $\textbf{u}.$ Moreover, the eigenvectors of (\ref{A_symbol}) are the columns of 
\begin{equation}
\label{eigenvectors_V}
U(\xi, \textbf{u})=\left(\begin{array}{ccccc}
0 & 0 & 0 & -\frac{B}{\sqrt{\gamma}} & \frac{B}{\sqrt{\gamma}} \\
-(1-B) & 0 & \frac{-B \xi_{2}}{|\xi|} & \frac{B(\gamma \xi_{1} - z_{1} \sqrt{\gamma} |\xi|)}{{\gamma} |\xi|} & \frac{B(\gamma \xi_{1} + z_{1} \sqrt{\gamma} |\xi|)}{{\gamma} |\xi|} \\
0 & -(1-B) & \frac{B \xi_{1}}{|\xi|} & \frac{B(\gamma \xi_{2} - z_{2} \sqrt{\gamma} |\xi|)}{{\gamma} |\xi|} & \frac{B(\gamma \xi_{2} + z_{2} \sqrt{\gamma} |\xi|)}{{\gamma} |\xi|} \\
1 & 0 & \frac{-\xi_{2}}{|\xi|} & \frac{\xi_{1}}{|\xi|} & \frac{\xi_{1}}{|\xi|} \\
0 & 1 & \frac{\xi_{1}}{|\xi|} & \frac{\xi_{2}}{|\xi|} & \frac{\xi_{2}}{|\xi|}
\end{array} \right),
\end{equation}

while its inverse matrix is
\begin{equation}
\label{eigenvectors_V_inverse}
U^{-1}(\xi, \textbf{u})=\left(\begin{array}{ccccc}
z_{1} & -1 & 0 & B & 0 \\
z_{2} & 0 & -1 & 0 & B \\
\frac{\xi_{2}z_{1}-\xi_{1}z_{2}}{|\xi|} & \frac{-\xi_{2}}{|\xi|} &  \frac{\xi_{1}}{|\xi|} & \frac{-(1-B)\xi_{2}}{|\xi|} & \frac{(1-B)\xi_{1}}{|\xi|} \\
-\frac{\sqrt{\gamma}|\xi|+B(z \cdot \xi)}{2B|\xi|} & \frac{\xi_{1}}{2|\xi|} & \frac{\xi_{2}}{2|\xi|} & \frac{(1-B)\xi_{1}}{2|\xi|} & \frac{(1-B)\xi_{2}}{2|\xi|} \\
\frac{\sqrt{\gamma} |\xi|-B(z \cdot \xi)}{2B|\xi|} & \frac{\xi_{1}}{2|\xi|} & \frac{\xi_{2}}{2|\xi|} & \frac{(1-B)\xi_{1}}{2|\xi|} & \frac{(1-B)\xi_{2}}{2|\xi|} \\
\end{array} \right).
\end{equation}

Since (\ref{eigenvectors_V}) and (\ref{eigenvectors_V_inverse}) are bounded for each $\xi \in \mathbb{R}^{2} - \{\textbf{0}\},$ the regularized version $\tilde{S}(\xi, \textbf{v})$ of the symbolic matrix 
\begin{equation}
\label{S}
S(\xi, \textbf{v}):=(U^{{-1}})^{{*}}U^{-1}(\xi, \textbf{v})
\end{equation}
can be associated to a \emph{Lax} symbolic symmetrizer $T_{\tilde{S}}$, see again \cite{Metivier}, \cite{Benzoni}, and \cite{Taylor} and the discussion in Section 3 below. Therefore, the hyperbolic part of (\ref{BL_System_Compact_NewVariables}) is symmetrizable. 
Unfortunately, again, the mere existence of a symmetrizer is not enough to get energy estimates for system (\ref{BL_System_Compact_NewVariables}), since we have to deal with the incompressible pressure term and then, by definition, with the projector operator in (\ref{P_operator}). As a matter of fact, the interaction between the symbolic symmetrizer $T_{\tilde{S}}$ and the gradient of the pressure term $\nabla P$ gives structural problems. Let us define the generalized projector operator:
\begin{equation}
\label{P_operator}
\textbf{P}(\xi):=\left( \begin{array}{ccc}
\textbf{1} & 0 & 0 \\
0 & \mathbb{P} & 0\\
0 & 0 & \textbf{1}\\
\end{array} \right),
\end{equation}
where $\mathbb{P}$ is the standard Leray projector, namely the projector onto the divergence free vector valued functions. The operators $T_{\tilde{S}}$ and $\textbf{P}$ do not commute and their commutator does not improve on the order  of the original symbols. Actually, if we apply the operator $\textbf{P}$ to system (\ref{BL_System_Compact_NewVariables}), with the aim of eliminating $\nabla P$, from (\ref{BL_System_Compact_NewVariables}) and (\ref{P_operator} we get
\begin{equation}
\label{Biofilms_Compact_Projector}
\partial_{t}\textbf{v}+\sum_{j=1}^{2}\textbf{P}\tilde{A}_{j}(\textbf{v})\partial_{x_{j}}\textbf{v}=\textbf{P}\tilde{G}(\textbf{v}),
\end{equation}

since $\textbf{P}\tilde{F}_{P}=(0, \mathbb{P} \nabla P, \textbf{0})=0$ by definition, and 
$$\textbf{P}\textbf{v}=(B, \mathbb{P}w, z)=(B, w, z),$$
by the divergence free condition $\nabla \cdot w=0$. By construction, $S$ symmetrizes $A(\xi, \textbf{v})$ in (\ref{A_symbol}), then we write the previous equation as
$$\partial_{t}\textbf{v}+\sum_{j=1}^{2}\tilde{A}_{j}(\textbf{v})\partial_{x_{j}}\textbf{v}+\sum_{j=1}^{2}[\tilde{A}_{j}, \textbf{P}]\partial_{x_{j}}\textbf{v}=\textbf{P}\tilde{G}(\textbf{v})$$
and we apply $T_{\tilde{S}}$ to its paradifferential formulation. Unfortunately, from (\ref{P_operator}), (\ref{A_symbol}) and (\ref{Paradiff_total}), the first term of the symbolic commutator in $\sum_{j=1}^{2}[\tilde{A}_{j}, \textbf{P}]\partial_{x_{j}}\textbf{v}$ contains the following term of degree 1 in $\xi$:
$$\tilde{A}(\xi, \textbf{v})\textbf{P}(\xi)-\textbf{P}(\xi)\tilde{A}(\xi, \textbf{v}),$$
then the commutator between $\tilde{A}(\xi, \textbf{v})$ and $\textbf{P}$ is still a symbol of degree 1, and it is not symmetrized by $S$. On the other hand, if at first we symmetrize the system by using the paradifferential operator $T_{\tilde{S}}$, the pressure gives the term
$$T_{\tilde{S}}\tilde{F}_{P}.$$ After that, when we project the equation by applying $\textbf{P}$ to it, the latter term reads
$$\textbf{P}T_{\tilde{S}}\tilde{F}_{P}=[T_{\tilde{S}}, \textbf{P}]\tilde{F}_{P},$$
whose symbol contains the smoothed version of the following term of degree 1 in $\xi$:
$$S\textbf{P}(0, i\xi_{1}P, i\xi_{2}P, 0, 0)$$
$$=\frac{i}{2|\xi|^{2}}(P(z\cdot \xi)(2\xi_{1}^{2}+3\xi_{2}^{2}), 0, 0, P\xi_{1}(2B\xi_{1}^{2}+3B\xi_{2}^{2}-\xi_{2}^{2}),  P\xi_{2}(2B\xi_{1}^{2}+3B\xi_{2}^{2}-\xi_{2}^{2})),$$
which is still a first order operator, then neither we are able to get energy estimates because of the lack of regularity of $P$, as discussed before and shown in (\ref{pressure_elliptic_equation1}), nor to get rid of the pressure term $P$ by using the projector operator $\textbf{P}$.

To be complete, we point out that system (\ref{BL_System_Compact}) in the new variable $\textbf{v}=(B, w, z)$ has also a classical symmetrizer
\begin{equation}
\label{classical_symm_v}
A_{0}(\textbf{v})=\left(\begin{array}{ccccc}
\frac{\gamma}{B}+|z|^{2} & -z_{1} & -z_{2} & 0 & 0 \\
-z_{1} & 1 & 0 & 0 & 0 \\
-z_{2} & 0 & 1 & 0 & 0 \\
0 & 0 & 0 & B(1-B) & 0 \\
0 & 0 & 0 & 0 & B(1-B) \\
\end{array} \right),
\end{equation}

which is strictly positive for $z$ small enough and under some assumptions on $B,$ discussed in Remark \ref{translation}. Anyway, again, the classical symmetrizer (\ref{classical_symm_v}) is not compatible with the projector operator (\ref{P_operator}), and this can be seen by arguing as for the symbolic symmetrizer related to $T_{\tilde{S}}$.

\subsection{A new formulation}
Taking inspiration from our preliminary work \cite{Bianchini}, we aim to propose a different symmetrization strategy for our problem, to be able to estimate correctly the pressure term. We apply first the operator $\textbf{P}$ to system (\ref{BL_System_Compact_NewVariables}). Notice that the initial datum does not change by projection, since the initial average velocity $w_{0}$ is already a divergence free vector and then, applying $\textbf{P}$ to (\ref{initial_data_new_variables}), we have $$\textbf{P}\textbf{v}_{0}(x)=(B_{0}(x), \mathbb{P}w_{0}(x), {z}_{0}(x))=(B_{0}(x), w_{0}(x), {z}_{0}(x)).$$ Moreover, $\textbf{P}\tilde{F}_{P}$ $=(0, \mathbb{P} \nabla P, \textbf{0})$ $=\textbf{0}$, and the divergence free constraint $\nabla \cdot w =0$ in (\ref{BL_System_Compact_NewVariables}) is implicitly contained in system (\ref{Biofilms_Compact_Projector}).
According to (\ref{paradiff_withoout_proj}), we consider the paradifferential version of system (\ref{Biofilms_Compact_Projector}):

\begin{equation}
\label{paradifferential_system}
\partial_{t}\textbf{v}+\textbf{P}T_{i\tilde{A}(\xi, \textbf{v})}\textbf{v}=\textbf{P}T_{\tilde{G}(\textbf{v})} + \sum_{j=1}^{d}[{\textbf{P}T_{\tilde{A}_{j}(\textbf{v})} - \textbf{P}\tilde{A}_{j}(\textbf{v})]\partial_{x_{j}}\textbf{v}}-[\textbf{P}T_{\tilde{G}(\textbf{v})}-\textbf{P}\tilde{G}(\textbf{v})].
\end{equation}

As we will see in details in the next section, in (\ref{paradifferential_system}) there is only one operator of order 1, which is $\textbf{P}T_{i\tilde{A}(\xi, \textbf{v})}.$ We want now to show that it is strongly hyperbolic, so we can find a symbolic symmetrizer for it.
From Proposition 1.10 in \cite{Grenier}, the symbol associated to the composition is made by the sum over the multi-index $\alpha$ of terms of type
$$\partial_{\xi}^{\alpha}\textbf{P}D_{x}^{\alpha}\tilde{A}(\xi, \textbf{v}),$$
where $D_{x}=\frac{1}{i}\partial_{x}$. The expansion above implies that there is only one term of degree $1$ in $\xi$, which is given for $|\alpha|=0,$ namely $\textbf{P}(\xi)\tilde{A}(\xi, \textbf{v}).$ Thus, the symbol of $\textbf{P}T_{i\tilde{A}(\xi, \textbf{v})}$ can be written as 

\begin{equation*}
\textbf{P}(\xi)i\tilde{A}(\xi, \textbf{v})+R(\xi, \textbf{v})
\end{equation*}
\begin{equation*}
=i\left( \begin{array}{ccccc}
(w+(1-2B)z)\cdot \xi & B\xi_{1} & B\xi_{2} & B(1-B)\xi_{1} & B(1-B)\xi_{2} \\
\frac{\xi_{2}(1-2B)\mu_{1}}{|\xi|^{2}} & \frac{\xi_{2}\mu_{2}}{|\xi|^{2}} & \frac{-\xi_{2}\mu_{3}}{|\xi|^{2}} & \frac{B(1-B)\xi_{2}\mu_{4}}{|\xi|^{2}} & \frac{-B(1-B)\xi_{2}\mu_{5}}{|\xi|^{2}} \\
\frac{-\xi_{1}(1-2B)\mu_{1}}{|\xi|^{2}} & \frac{-\xi_{1}\mu_{2}}{|\xi|^{2}} & \frac{\xi_{1}\mu_{3}}{|\xi|^{2}} & \frac{-B(1-B)\xi_{1}\mu_{4}}{|\xi|^{2}} & \frac{B(1-B)\xi_{1}\mu_{5}}{|\xi|^{2}} \\
\frac{\gamma \xi_{1}}{B}-z_{1}(z \cdot \xi) & z \cdot \xi & 0 & (w+(1-2B)) \cdot \xi & 0 \\
\frac{\gamma \xi_{2}}{B}-z_{2}(z \cdot \xi) & 0 & z \cdot \xi & 0 & (w+(1-2B)z) \cdot \xi \\
\end{array} \right)
\end{equation*}
\begin{equation}
\label{PA_symbol}
+R(\xi, \textbf{v}),
\end{equation}
where $R(\xi, \textbf{v})$ is a remainder of order less than or equal to $0,$ and
\begin{equation}
\label{parameters}
\begin{cases}
& \mu_{1}:=(z \cdot \xi)(\xi_{2}z_{1}-\xi_{1}z_{2}), \\
& \mu_{2}:=(w \cdot \xi) \xi_{2} + B\xi_{1} (\xi_{2}z_{1}-\xi_{1}z_{2}), \\
& \mu_{3}:=(w \cdot \xi) \xi_{1} + B\xi_{2} (\xi_{2}z_{1}-\xi_{1}z_{2}), \\
& \mu_{4}:=z_{2}(\xi_{2}^{2}-\xi_{1}^{2})+2z_{1}\xi_{1}\xi_{2},\\
& \mu_{5}:=z_{1}(\xi_{1}^{2}-\xi_{2}^{2})+2z_{2}\xi_{1}\xi_{2}.\\
\end{cases}
\end{equation}
The eigenvalues of $\textbf{P}\tilde{A}(\xi, \textbf{v})$ are the following:

\begin{equation}
\label{eigenvalues_PA}
\begin{array}{ccc}
\lambda_{1}=0,& \lambda_{2}=(w-Bz) \cdot \xi, & \lambda_{3}=(w+(1-B)z) \cdot \xi, \\
\end{array}
\end{equation}
\begin{equation}
\lambda_{4/5}=(w+(1-2B)z) \cdot \xi \pm \sqrt{(1-B)\Delta_{2}}.\\
\end{equation}

Its eigenvectors are the columns of $V(\xi, \textbf{v})$
\begin{equation}
\label{matrix_V}
=\left( \begin{array}{ccccc}
\frac{B|\xi|(w-Bz) \cdot \xi}{\Delta_{1}} & 0 & 0 & \frac{-B|\xi| \sqrt{1-B}}{\sqrt{\Delta_{2}}} &  \frac{B|\xi| \sqrt{1-B}}{\sqrt{\Delta_{2}}} \\
\frac{p_{1}}{|\xi| \Delta_{1}} & \frac{(1-B)\xi_{2}}{|\xi|} & \frac{-B\xi_{2}}{|\xi|} & \frac{B\xi_{2}(\xi_{1}z_{2}-\xi_{2}z_{1})\sqrt{1-B}}{|\xi|\sqrt{\Delta_{2}}} & \frac{-B\xi_{2}(\xi_{1}z_{2}-\xi_{2}z_{1})\sqrt{1-B}}{|\xi|\sqrt{\Delta_{2}}} \\
\frac{p_{2}}{|\xi| \Delta_{1}} & \frac{-(1-B)\xi_{1}}{|\xi|} &  \frac{B\xi_{1}}{|\xi|} & \frac{-B\xi_{1}(\xi_{1}z_{2}-\xi_{2}z_{1})\sqrt{1-B}}{|\xi|\sqrt{\Delta_{2}}} & \frac{B\xi_{1}(\xi_{1}z_{2}-\xi_{2}z_{1})\sqrt{1-B}}{|\xi|\sqrt{\Delta_{2}}} \\
\frac{\xi_{1}}{|\xi|} & \frac{-\xi_{2}}{|\xi|} & \frac{-\xi_{2}}{|\xi|} &  \frac{\xi_{1}}{|\xi|} &  \frac{\xi_{1}}{|\xi|} \\
 \frac{\xi_{2}}{|\xi|} & \frac{\xi_{1}}{|\xi|} &  \frac{\xi_{1}}{|\xi|} &  \frac{\xi_{2}}{|\xi|} &  \frac{\xi_{2}}{|\xi|} \\
\end{array} \right),
\end{equation}
where $p_{1}=p_{1}(\xi, \textbf{v}), p_{2}=p_{2}(\xi, \textbf{v})$ are polynomial functions of degree 3 in $\xi$ depending on $\textbf{v},$ and
\begin{equation}
\label{Delta}
\begin{array}{cc}
\Delta_{1}:=(1-B)(z \cdot \xi)^{2}-\gamma |\xi|^{2}+(w \cdot \xi)(z \cdot \xi), & \Delta_{2}:=\gamma |\xi|^{2}-B(z \cdot \xi)^{2}. 
\end{array}
\end{equation}

Its inverse matrix $V^{-1}(\xi, \textbf{v})$
\begin{equation}
\label{inv}
=\left( \begin{array}{ccccc} 
0 & \frac{-\xi_{1}\Delta_{1}}{|\xi|\Delta_{3}} & \frac{-\xi_{2}\Delta_{1}}{|\xi|\Delta_{3}} & 0 & 0 \\
\frac{\xi_{1}z_{2}-\xi_{2}z_{1}}{|\xi|} & \frac{\xi_{2}}{|\xi|} & \frac{-\xi_{1}}{|\xi|} &  \frac{-B\xi_{2}}{|\xi|} & \frac{B\xi_{1}}{|\xi|} \\
\frac{-\xi_{1}z_{2}+\xi_{2}z_{1}}{|\xi|} & \frac{-\xi_{2}}{|\xi|} & \frac{\xi_{1}}{|\xi|} & \frac{-(1-B)\xi_{2}}{|\xi|} & \frac{(1-B)\xi_{1}}{|\xi|} \\
\frac{-\sqrt{\Delta_{2}}}{2B |\xi|\sqrt{1-B}}&\frac{\xi_{1}q_{1}}{2|\xi| \Delta_{3}\sqrt{(1-B)\Delta_{2}}}&\frac{\xi_{2}q_{1}}{2|\xi| \Delta_{3}\sqrt{(1-B)\Delta_{2}}}&\frac{\xi_{1}}{2 |\xi|}&\frac{\xi_{2}}{2 |\xi|}\\
\frac{\sqrt{\Delta_{2}}}{2B |\xi| \sqrt{(1-B)}}&\frac{\xi_{1} q_{2}}{2|\xi|\Delta_{3}\sqrt{(1-B)\Delta_{2}}}&\frac{\xi_{2} q_{2}}{2|\xi|\Delta_{3}\sqrt{(1-B)\Delta_{2}}}&\frac{\xi_{1}}{2 |\xi|}&\frac{\xi_{2}}{2 |\xi|}\\
\end{array} \right),
\end{equation}
where $q_{1}=q_{1}(\xi, \textbf{v}), q_{2}=q_{2}(\xi, \textbf{v})$ are polynomial functions of degree 3 in $\xi$ and
\begin{equation}
\label{Delta3}
\Delta_{3}:=(1-3B(1-B)) (z \cdot \xi)^{2} + (w \cdot \xi)^{2} - \gamma (1-B) |\xi|^{2} + 2(1-2B)(w \cdot \xi) (z \cdot \xi). 
\end{equation}

\begin{proposition}
\label{assumptions_prop}
Under the following assumptions
\begin{equation}
\label{assumptions}
\Delta_{1} \neq 0,  ~~~ \Delta_{2} > 0 ~~~\text{and}~~~ \Delta_{3} \neq 0 ~~~ \text{for} ~~~  \xi \neq (0, 0),
\end{equation}
the first order operator of system (\ref{paradifferential_system}) is {strongly hyperbolic}.
\end{proposition}
\begin{proof}
Considering the symbolic matrix (\ref{PA_symbol}) and the related eigenvalues in (\ref{eigenvalues_PA}) and eigenvectors in (\ref{matrix_V}), it follows by the definition of strong hyperbolicity, see \cite{Metivier}.
\end{proof}

\begin{proposition}
\label{control}
Under the following conditions
\begin{equation}
\small
\label{Delta_assumptions}
\begin{cases}
& 2 \gamma>(1-B) |z|^{2}+(w \cdot z), \\
& \gamma^{2} > \gamma (1-B)|z|^{2}+\gamma (w \cdot z) + \frac{w_{1}z_{2}^{2}}{4}+\frac{w_{2}z_{1}^{2}}{4}, \\
& \gamma > B|z|^{2}, \\
& 2 \gamma > B|z|^{2}, \\
& 2\gamma(1-B) > (1-3B(1-B)) |z|^{2}+2(1-2B)(w \cdot z)+|w|^{2}, \\
& \gamma^{2}(1-B)^{2} > \gamma (1-B) ((1-3B(1-B))|z|^{2}+|w|^{2}+2(1-2B)(w \cdot z)) \\
& -((1-3B(1-B))z_{1}^{2}+2w_{1}z_{1}(1-2B)+w_{1}^{2})((1-3B(1-B))z_{2}^{2}+2w_{2}z_{2}(1-2B)+w_{2}^{2}) \\
& +((1-3B(1-B))z_{1}z_{2}+(1-2B)(w_{1}z_{2}+w_{2}z_{1})+w_{1}w_{2})^{2}, \\
\end{cases}
\end{equation}
the value $\xi=(0,0)$ is a strict minimum, maximum and minimum point for $\Delta_{1}, \Delta_{2}$ and $\Delta_{3}$ respectively and $\Delta_{1}|_{\xi_{1}=\xi_{2}=0}=\Delta_{2}|_{\xi_{1}=\xi_{2}=0}=\Delta_{3}|_{\xi_{1}=\xi_{2}=0}=0.$ Therefore, Proposition \ref{assumptions_prop} is verified.
\end{proposition}

\begin{proposition}
\label{equilibrium_point_prop}
For any $\textbf{v}=(B, w, z)$ in a small neighborhood of the equilibrium point $\bar{\textbf{v}}=\phi(\bar{\textbf{u}})$, with $\phi$ in (\ref{diffeomorphism}) and $\bar{\textbf{u}}$ in (\ref{equilibrium_point1}), the first order operator of system (\ref{paradifferential_system}) is strongly hyperbolic.
\end{proposition}
\begin{proof}
It follows directly from Proposition \ref{control} and Proposition \ref{assumptions_prop}.
\end{proof}

\section{Main result}

In this section we prove a local existence result in the Sobolev spaces for the Cauchy problem associated to (\ref{paradifferential_system}). To the  best of our knowledge, such a result is not explicitly stated in all the relevant works about paradifferential calculus. For instance, in the lecture notes \cite{Metivier}, only linear and quasi-linear equations of differential operators are considered, while in \cite{Grenier} the discussion is extended to evolution equations of pseudodifferential operators, but the proof makes use of some particular structural characteristics that our system does not satisfy. Therefore, we give our proof of existence and uniqueness of the solution to the Cauchy problem associated to the translated version of (\ref{paradifferential_system}). 
We state here our main result. Since we will work in Sobolev spaces, we define
\begin{equation}
\label{Vs}
V^{s}:=\{\textbf{v}=(B, w, z) \in H^{s}(\mathbb{R}^{2}) | \nabla \cdot w=0\}.
\end{equation}
\begin{theorem}
\label{Convergence}
Let $\tilde{\textbf{v}}_{0}:=\textbf{v}_{0}-\bar{\textbf{v}}$, with $\textbf{v}_{0}$ in (\ref{initial_data_new_variables}), $\bar{\textbf{v}}=\phi(\bar{\textbf{u}})$ in (\ref{equilibrium_point1}), and $\tilde{\textbf{v}}_{0} \in V^{s}$ with $s > 2.$ There is a positive time $T$, such that there exists the unique $\tilde{\textbf{v}} \in C([0,T], V^{s}) \cap C^{1}([0,T], V^{s})$ and a function $P$ such that $\nabla P \in C([0,T], H^{s-1}(\mathbb{R}^{2}))$ which solve
\begin{equation}
\label{BL_compact_translated}
\begin{cases}
& \partial_{t}\tilde{\textbf{v}}+\sum_{j=1}^{2}\tilde{A}_{j}((\tilde{\textbf{v}}+\bar{\textbf{v}}))\partial_{x_{j}}\tilde{\textbf{v}}+\tilde{F}_{P}
=\tilde{G}((\tilde{\textbf{v}}+\bar{\textbf{v}})), \\
& \nabla \cdot \tilde{w}=0. \\
\end{cases}
\end{equation}
The solution $(\tilde{\textbf{v}}, P)$ to (\ref{BL_compact_translated}) is the limit of the sequence of the solutions to the approximating system (\ref{Approximation}) below, with initial data (\ref{initial_data_approximating}).
\end{theorem}
The proof follows by combining in a classical ways, see for instance \cite{Bertozzi}, theorems \ref{Local Existence of Approximating Solutions_P} and \ref{unique} below. First, following  \cite{Bertozzi}, we write our approximation to system (\ref{BL_compact_translated}) via a regularization of the operator and the Picard iterations. Namely, let $J_{\varepsilon}$ be a standard mollifier, then solve
\begin{equation}
\label{Approximation}
\begin{cases}
& \partial_{t}\tilde{\textbf{v}}^{\varepsilon}+\sum_{j=1}^{2}J_{\varepsilon}\tilde{A}_{j}(J_{\varepsilon}(\tilde{\textbf{v}}^{\varepsilon}+\bar{\textbf{v}}))\partial_{x_{j}}J_{\varepsilon}\tilde{\textbf{v}}^{\varepsilon}+\tilde{F}_{P}^{\varepsilon}
=J_{\varepsilon}\tilde{G}(J_{\varepsilon}(\tilde{\textbf{v}}^{\varepsilon}+\bar{\textbf{v}})), \\
& \nabla \cdot \tilde{w}^{\varepsilon}=0, \\
\end{cases}
\end{equation}
where $\tilde{F}_{P}^{\varepsilon}=(0, \nabla P^{\varepsilon}, \textbf{0}),$ the initial data are
\begin{equation}
\label{initial_data_approximating}
\tilde{\textbf{v}}^{\varepsilon}(0,x)=\tilde{\textbf{v}}^{\varepsilon}_{0}(x)=(\tilde{B}_{0}^{\varepsilon}, \tilde{w}_{0}^{\varepsilon}, \tilde{z}_{0}^{\varepsilon})=\tilde{\textbf{v}}_{0}:=\textbf{v}_{0}-\bar{\textbf{v}},
\end{equation}
and $\tilde{\textbf{v}}_{0}$ as in Theorem \ref{Convergence}. We apply now $\textbf{P}$ to (\ref{Approximation}) to get the projected version

\begin{equation}
\label{P_approx_compact}
\partial_{t}\tilde{\textbf{v}}^{\varepsilon}+\sum_{j=1}^{2}\textbf{P}J_{\varepsilon}\tilde{A}_{j}(J_{\varepsilon}(\tilde{\textbf{v}}^{\varepsilon}+\bar{\textbf{v}}))\partial_{x_{j}}J_{\varepsilon}\tilde{\textbf{v}}^{\varepsilon}
=\textbf{P}J_{\varepsilon}\tilde{G}(J_{\varepsilon}(\tilde{\textbf{v}}^{\varepsilon}+\bar{\textbf{v}})), \\
\end{equation}
with initial data in (\ref{initial_data_approximating}). 
\begin{theorem}(Local existence of the approximating solution)
\label{Local Existence of Approximating Solutions_P}
Let $\tilde{\textbf{v}}^{\varepsilon}_{0}=(\tilde{B}_{0}^{\varepsilon}, \tilde{w}_{0}^{\varepsilon}, \tilde{z}_{0}^{\varepsilon}) \in V^{s}$ in (\ref{initial_data_approximating}), with $s > 2.$ Then, for every $\varepsilon > 0$, there exists a time $T$, independent of $\varepsilon$, such that system (\ref{P_approx_compact}) has a unique solution $\tilde{\textbf{v}}^{\varepsilon}=(\tilde{B}^{\varepsilon}, \tilde{w}^{\varepsilon}, \tilde{z}^{\varepsilon})\in C^{1}([0, T], V^{s})$.
\end{theorem}

\begin{proof}
First, we show that existence and uniqueness follow from the Picard theorem (see \cite{Bertozzi}). System (\ref{P_approx_compact}) can be reduced to an ordinary differential equation
\begin{equation}
\label{ODE_P}
\partial_{t}\tilde{\textbf{v}}^{\varepsilon} = F^{\varepsilon}(\tilde{\textbf{v}}^{\varepsilon}),  ~~~~~~~~~~\tilde{\textbf{v}}^{\varepsilon}(0,x)=\tilde{\textbf{v}}^{\varepsilon}_{0}(x),
\end{equation}
where 
\begin{equation}
\label{Function_ODE_P}
F^{\varepsilon}(\tilde{\textbf{v}}^{\varepsilon}) = -\sum_{j=1}^{2}\textbf{P} J_{\varepsilon}\tilde{A}_{j}(J_{\varepsilon}(\tilde{\textbf{v}^{\varepsilon}}+\bar{\textbf{v}}))\partial_{x_{j}}J_{\varepsilon}\tilde{\textbf{v}}^{\varepsilon} + \textbf{P}J_{\varepsilon}\tilde{G}(J_{\varepsilon}(\tilde{\textbf{v}}^{\varepsilon}+\bar{\textbf{v}})) =: F_{1}^{\varepsilon}(\tilde{\textbf{v}}^{\varepsilon})+F_{2}^{\varepsilon}(\tilde{\textbf{v}}^{\varepsilon}).
\end{equation}
Notice that $J_{\varepsilon}\tilde{\textbf{v}}^{\varepsilon}$ and $J_{\varepsilon}(\tilde{\textbf{v}}^{\varepsilon}+\bar{\textbf{v}})$ are $C^{\infty}$ functions and, from \cite{Metivier}, $\mathbb{P}$ is associated to an analytic pseudo-differential operator of order 0, modulo an infinitely smooth remainder, so that
$$F^{\varepsilon}: V^{s} \rightarrow V^{s}.$$ 
In order to apply the Picard theorem, we have to prove that $F^{\varepsilon}(\tilde{\textbf{v}}^{\varepsilon})$ in (\ref{Function_ODE_P}) is Lipschitz continuous. To do this, we take two vectors $\tilde{\textbf{v}}_{1}, \tilde{\textbf{v}}_{2}$ in $V^{s}$. In the following, we omit the index $\varepsilon$ in the unknown functions, where there is no ambiguity. Using Theorem C.12 in \cite{Benzoni}, it is straightforward here to prove that

\begin{equation}
\label{last_first_appr}
||F_{1}^{\varepsilon}(\tilde{\textbf{v}}_{1})-F_{1}^{\varepsilon}(\tilde{\textbf{v}}_{2})||_{s}  \le c(c_{S}, ||\tilde{\textbf{v}}_{1}||_{s}, ||\tilde{\textbf{v}}_{2}||_{s}, \bar{B}, \varepsilon^{-1}) ||\tilde{\textbf{v}}_{1}-\tilde{\textbf{v}}_{2}||_{s},
\end{equation}

where $c_{S}$ is the Sobolev embedding constant and the last inequality follows from Moser estimates and properties of mollifiers. Similarly, we have 

\begin{equation}
\label{last_first_appr1}
||F_{2}^{\varepsilon}(\tilde{\textbf{v}}_{1})-F_{2}^{\varepsilon}(\tilde{\textbf{v}}_{2})||_{s}  \le c(c_{S},||\tilde{\textbf{v}}_{1}||_{s}, ||\tilde{\textbf{v}}_{2}||_{s}, \bar{B}) ||\tilde{\textbf{v}}_{1}-\tilde{\textbf{v}}_{2}||_{s}.
\end{equation}

From (\ref{last_first_appr}) and (\ref{last_first_appr1}) we have that, for fixed $\varepsilon$, $F^{\varepsilon}$ is locally Lipschitz continuous on any open set
\begin{equation}
\mathcal{U}^{M}=\{\tilde{\textbf{v}}^{\varepsilon} \in V^{s} : ||\tilde{\textbf{v}}^{\varepsilon}||_{s} \le M \}.
\end{equation}
Then, the Picard theorem provides a unique solution $\tilde{\textbf{v}}^{ \varepsilon} \in C^{1}([0, T_{\varepsilon}), \mathcal{U}^{M})$ for any $T_{\varepsilon} > 0$.

Now, we want to show that the time of existence $T_{\varepsilon}$ is bounded from below by any strictly positive time $T$ that is independent of $\varepsilon$. 

As we pointed out in the Introduction, the main problem with our original system \eqref{Approximation} is that it is difficult to give for it a direct energy estimate, since the pressure term is not well behaved against both the symmetrizers, the classical one (\ref{A0}) or the \emph{Lax} one  (\ref{S}),  that work only on the hyperbolic part of system \eqref{Approximation}, disregarding the pressure. However, in Section 2 we noticed that system \eqref{P_approx_compact} is strongly hyperbolic near the equilibrium point $\bar{\textbf{v}}=\phi(\bar{\textbf{u}}),$ in (\ref{equilibrium_point1}), and so we can
try to find an appropriate symmetrizer for this system, which in this case is forced to be a paradifferential operator. Our construction in the following is essentially based on the techniques developed in \cite{Metivier}, which are combined to the ideas in \cite{Grenier} and adapted to our specific operator.

According to (\ref{paradifferential_system}), from (\ref{P_approx_compact}) we have
\begin{equation*}
\partial_{t}\tilde{\textbf{v}}^{\varepsilon}+\textbf{P}J_{\varepsilon}T_{i\tilde{A}(\xi, J_{\varepsilon}(\tilde{\textbf{v}}^{\varepsilon}+\bar{\textbf{v}}))}J_{\varepsilon}\tilde{\textbf{v}}^{\varepsilon}=\sum_{j=1}^{2}\textbf{P}J_{\varepsilon}[T_{\tilde{A}_{j}(J_{\varepsilon}(\tilde{\textbf{v}}^{\varepsilon}+\bar{\textbf{v}}))}-\tilde{A}_{j}(J_{\varepsilon}(\tilde{\textbf{v}}^{\varepsilon}+\bar{\textbf{v}}))]\partial_{x_{j}}J_{\varepsilon}\tilde{\textbf{v}}^{\varepsilon}
\end{equation*}

\begin{equation}
\label{paradiff_approx}
+\textbf{P}J_{\varepsilon}T_{\tilde{G}(J_{\varepsilon}(\tilde{\textbf{v}}^{\varepsilon}+\bar{\textbf{v}}))}-\textbf{P}J_{\varepsilon}[T_{\tilde{G}(J_{\varepsilon}(\tilde{\textbf{v}}^{\varepsilon}+\bar{\textbf{v}}))}-\tilde{G}(J_{\varepsilon}(\tilde{\textbf{v}}^{\varepsilon}+\bar{\textbf{v}}))].
\end{equation}
From Lemma 7.2.3 in \cite{Metivier}, properties of mollifiers and the Leray projector, we get
\begin{equation}
\label{remainder_estimate}
\begin{array}{l}
||\textbf{P}J_{\varepsilon}\{[T_{\tilde{A}_{j}(J_{\varepsilon}(\tilde{\textbf{v}}^{\varepsilon}+\bar{\textbf{v}}))} - \tilde{A}_{j}(J_{\varepsilon}(\tilde{\textbf{v}}^{\varepsilon}+\bar{\textbf{v}}))]\partial_{x_{j}}J_{\varepsilon}\tilde{\textbf{v}}^{\varepsilon}\}||_{s} 
\le c(||\tilde{\textbf{v}}^{\varepsilon}||_{s}, \bar{B}) ||\tilde{\textbf{v}}^{\varepsilon}||_{s},  ~ \text{and}\\ \\
||\textbf{P}J_{\varepsilon}\{T_{\tilde{G}(J_{\varepsilon}(\tilde{\textbf{v}}^{\varepsilon}+\bar{\textbf{v}}))}-\tilde{G}(J_{\varepsilon}(\tilde{\textbf{v}}^{\varepsilon}+\bar{\textbf{v}}))\}||_{s} \le c(||\tilde{\textbf{v}}^{\varepsilon}||_{s}, \bar{B}) ||\tilde{\textbf{v}}^{\varepsilon}||_{s}.
\end{array}
\end{equation}

Then, we can focus on the paradifferential part of (\ref{paradiff_approx}), which is
\begin{equation}
\label{pure_paradiff}
\partial_{t}\tilde{\textbf{v}}^{\varepsilon}+J_{\varepsilon}\textbf{P}T_{i\tilde{A}(\xi, J_{\varepsilon}(\tilde{\textbf{v}}^{\varepsilon}+\bar{\textbf{v}}))}J_{\varepsilon}\tilde{\textbf{v}}^{\varepsilon}-J_{\varepsilon}\textbf{P}T_{\tilde{G}(J_{\varepsilon}(\tilde{\textbf{v}}^{\varepsilon}+\bar{\textbf{v}}))}.
\end{equation}

From \eqref{PA_symbol}, we know that the symbolic matrix associated to the composition $\textbf{P}T_{i\tilde{A}(\xi, J_{\varepsilon}(\textbf{v}^{\varepsilon}+\bar{\textbf{v}}))}$ can be written as

\begin{equation}
\label{symbol_composition}
\begin{array}{c}
\textbf{P}(\xi)i\tilde{A}(\xi, J_{\varepsilon}(\tilde{\textbf{v}}^{\varepsilon}+\bar{\textbf{v}}))+R, 
\end{array}
\end{equation}

where $\textbf{P}(\xi)\tilde{A}(\xi, J_{\varepsilon}(\tilde{\textbf{v}}^{\varepsilon}+\bar{\textbf{v}}))$ is the symbolic part of degree 1, while $R$ is a remainder of order less than or equal to 0. Now, by construction 

\begin{equation}
\label{diagonalization}
\textbf{P}(\xi)\tilde{A}=V D V^{-1},
\end{equation}
with $D$ the diagonal matrix of the eigenvalues of $\textbf{P}\tilde{A}$ in (\ref{eigenvalues_PA}), namely 
$$(V^{-1})^{*}V^{-1}\textbf{P}\tilde{A}=(V^{-1})^{*} D V^{-1}$$ is symmetric. This way,
setting
\begin{equation}
\label{W}
W(\xi, J_{\varepsilon}(\tilde{\textbf{v}}^{\varepsilon}+\bar{\textbf{v}})):=(1-\theta_{\lambda}(\xi))V^{-1}(\xi, J_{\varepsilon}(\tilde{\textbf{v}}^{\varepsilon}+\bar{\textbf{v}})),
\end{equation}
with $V^{-1}(\xi, J_{\varepsilon}(\tilde{\textbf{v}}^{\varepsilon}+\bar{\textbf{v}}))$ in (\ref{inv}) and, following \emph{M\'etivier} in \cite{Metivier}, by using
\begin{equation}
\label{theta_lambda}
\theta_{\lambda}(\xi)Id=\theta(\lambda^{-1} \xi)Id,
\end{equation}
for any fixed parameter $\lambda$ and for any $\theta(\xi) \in C_{c}^{\infty}(\mathbb{R}^{2})$ such that $0 \le \theta \le 1$ for $1 < |\xi| < 2 $,  $\theta=1$ for $|\xi| \le 1$ and $\theta=0$ for $|\xi|\ge 2$, we define the regularized symmetrizer
\begin{equation}
\label{Lax_symmetrizer}
\Sigma:=({T_{W}})^{*}T_{W}+\theta_{\lambda}^{2}(D_{x})Id,
\end{equation}

where ${(T_{W})^{*}}$ is the conjugate of the paradifferential operator ${T_{{W}}}$ associated to (\ref{W}).  Thus, by construction, $\Sigma$ in (\ref{Lax_symmetrizer}) is symmetric. Moreover,
\begin{equation}
\label{symmetrizer_norm}
(\Sigma \textbf{u}, \textbf{u})_{0} = ||T_{W}\textbf{u}||_{0}^{2} + ||\theta_{\lambda}(D_{x})\textbf{u}||_{0}^{2},
\end{equation}
for every $\textbf{u} \in L^{2}(\mathbb{R}^{2})$. Now, we have to show the equivalence of (\ref{symmetrizer_norm}) with respect to the $L^{2}$-norm. We prove the following lemma, which is an adapted version of Lemma 7.1.6 in \cite{Metivier}, where we replace the square root of a more general \emph{Lax}-symmetrize with $V^{-1}$ in (\ref{inv}).

\begin{lemma}
There exist constant values $\bar{c}, \underline{c}$ such that, for every $\textbf{u} \in L^{2}(\mathbb{R}^{2})$, we have
\begin{equation}
\label{lemma}
\underline{c} ||\textbf{u}||_{0}^{2} \le (\Sigma \textbf{u}, \textbf{u})_{0} \le \bar{c} ||\textbf{u}||_{0}^{2}.
\end{equation}
\end{lemma}

\begin{proof}
Since $\Sigma$ in (\ref{Lax_symmetrizer}) is an operator of order 0, the right side follows directly from paradifferential properties (see \cite{Metivier}). We focus on the left one. 

Let $W_{1}:=(1-\theta(\xi))V^{-1}(\xi, \textbf{u})$ and $W_{2}:=(1-\theta(\xi))V(\xi, \textbf{u}).$ By construction,

$$W_{2}W_{1}=(1-\theta(\xi))^{2}Id.$$ Notice that, for $\lambda \ge 2,$ from (\ref{theta_lambda}) we have 

$$(1-\theta_{\lambda}(\xi))(1-\theta(\xi))=(1-\theta_{\lambda}(\xi)).$$
From (\ref{W}), it yields

$$(1-\theta(\xi))(1-\theta_{\lambda}(\xi))V^{-1}(\xi, \textbf{u})$$

$$=(1-\theta_{\lambda}(\xi))V^{-1}(\xi, \textbf{u})=W,$$ and

$$W_{2}W=(1-\theta(\xi))(1-\theta_{\lambda}(\xi))Id = (1-\theta_{\lambda}(\xi))Id.$$

From the composition theorem in \cite{Metivier}, we have
$$T_{W_{2}}T_{W}=(Id+R)(1-\theta_{\lambda}(D_{x})),$$
where $R$ is a remainder of order less than or equal to -1. In particular,

\begin{equation}
\label{-1norm}
||(1-\theta_{\lambda}(D_{x}))\textbf{u}||_{0} \le c(W_{2})||T_{W}\textbf{u}||_{0}+c(R)||(1-\theta_{\lambda}(D_{x}))\textbf{u}||_{H^{-1}},
\end{equation}
for every $\textbf{u} \in L^{2}(\mathbb{R}^{2})$. Now, setting $\Lambda(\xi)=(1-\Delta(\xi))^{\frac{1}{2}}$, where $\Delta(\xi)$ is the Laplace operator, 

\begin{equation}
\label{lambda2}
||(1-\theta_{\lambda}(D_{x}))\textbf{u}||_{H^{-1}}=||(1-\theta_{\lambda}(\xi))\Lambda^{-1}(\xi)\hat{\textbf{u}}||_{0}.
\end{equation}

From (\ref{theta_lambda}), $\displaystyle (1-\theta_{\lambda}(\xi))\Lambda(\xi)^{-1} = \frac{(1-\theta_{\lambda}(\xi))}{(1+|\xi|^{2})^{1/2}} \le \frac{1}{\lambda},$
and, from, (\ref{lambda2}), 

$$||(1-\theta_{\lambda}(D_{x}))\textbf{u}||_{H^{-1}} \le \frac{1}{\lambda} ||(1-\theta_{\lambda}(D_{x}))\textbf{u}||_{0}.$$

This gives
$$||(1-\theta_{\lambda}(D_{x}))\textbf{u}||_{0} \le c(W_{2})||T_{W}\textbf{u}||_{0}+ \frac{c(R)}{\lambda}||(1-\theta_{\lambda}(D_{x}))\textbf{u}||_{0},$$

then we can choose the parameter $\lambda \ge 2$ big enough such that $\frac{c(R)}{\lambda} < 1$. This way,

$$||(1-\theta_{\lambda}(D_{x}))\textbf{u}||_{0} \le c(W_{2})||T_{W}\textbf{u}||_{0}.$$

Squaring, we have

\begin{equation}
||\textbf{u}||_{0}^{2} \le c(W_{2}) ||T_{W}\textbf{u}||_{0}^{2}+||\theta_{\lambda}\textbf{u}||_{0}^{2}.
\end{equation}
\end{proof}

Now, we are ready to get energy estimates. Applying $\Lambda^{s}$ and the symmetrizer (\ref{Lax_symmetrizer}) to (\ref{paradiff_approx}), we have

\begin{equation}
\label{s_energy_estimate}
\frac{d}{dt}(\Sigma \Lambda^{s} \tilde{\textbf{v}}^{\varepsilon}, \Lambda^{s}\tilde{\textbf{v}}^{\varepsilon})_{0} = (\partial_{t}\Sigma \Lambda^{s}\tilde{\textbf{v}}^{\varepsilon}, \Lambda^{s}\tilde{\textbf{v}}^{\varepsilon})_{0} + 2 (\Sigma \Lambda^{s}\partial_{t}\tilde{\textbf{v}}^{\varepsilon}, \Lambda^{s}\tilde{\textbf{v}}^{\varepsilon})_{0}.
\end{equation}

The operator of the first term of the right-hand side,
\begin{equation}
\partial_{t}\Sigma={T_{\partial_{t}W}}^{*}T_{W}+{T_{W}}^{*}T_{\partial_{t}W},
\end{equation}
has order 0 and depends on $\partial_{t}\tilde{\textbf{v}}^{\varepsilon},$ i.e.
\begin{equation}
|(\partial_{t}\Sigma \Lambda^{s}\tilde{\textbf{v}}^{\varepsilon}, \Lambda^{s}\tilde{\textbf{v}}^{\varepsilon})_{0}| \le c(|\partial_{t}\tilde{\textbf{v}}^{\varepsilon}|_{\infty})||\tilde{\textbf{v}}^{\varepsilon}||_{s}^{2}\le c(|\tilde{\textbf{v}}^{\varepsilon}|_{\infty}, |\partial_{x_{j}}\tilde{\textbf{v}}^{\varepsilon}|_{\infty})||\tilde{\textbf{v}}^{\varepsilon}||_{s}^{2} \le c(||\tilde{\textbf{v}}^{\varepsilon}||_{s}) ||\tilde{\textbf{v}}^{\varepsilon}||_{s}^{2}
\end{equation}

where the inequalities follow from (\ref{paradiff_approx}) and the Sobolev embedding theorem. The last term of (\ref{s_energy_estimate}) yields

$$(\Sigma \Lambda^{s}\partial_{t}\tilde{\textbf{v}}^{\varepsilon}, \Lambda^{s}\tilde{\textbf{v}}^{\varepsilon})_{0}=-Re(\Sigma \Lambda^{s} \textbf{P}J_{\varepsilon}T_{i\tilde{A}(\xi, J_{\varepsilon}(\tilde{\textbf{v}}^{\varepsilon}+\bar{\textbf{v}}))}J_{\varepsilon}\tilde{\textbf{v}}^{\varepsilon}, \Lambda^{s} \tilde{\textbf{v}}^{\varepsilon})_{0}$$
$$+(\Sigma\Lambda^{s} \textbf{P}J_{\varepsilon} T_{\tilde{G}(J_{\varepsilon}(\tilde{\textbf{v}}^{\varepsilon}+\bar{\textbf{v}}))}, \Lambda^{s} \tilde{\textbf{v}}^{\varepsilon})_{0} + Q^{\varepsilon},$$

where

$$Q^{\varepsilon}=\sum_{j=1}^{2}(\Lambda^{s} \textbf{P}J_{\varepsilon}[T_{\tilde{A}_{j}(J_{\varepsilon}(\textbf{v}^{\varepsilon}+\bar{\textbf{v}}))}-\tilde{A}_{j}(J_{\varepsilon}(\textbf{v}^{\varepsilon}+\bar{\textbf{v}}))]\partial_{x_{j}}J_{\varepsilon}\textbf{v}^{\varepsilon}, \Lambda^{s}\tilde{\textbf{v}}^{\varepsilon})_{0}$$

\begin{equation}
- (\Lambda^{s}\textbf{P}J_{\varepsilon}[T_{\tilde{G}(J_{\varepsilon}(\textbf{v}^{\varepsilon}+\bar{\textbf{v}}))}-\tilde{G}(J_{\varepsilon}(\textbf{v}^{\varepsilon}+\bar{\textbf{v}}))], \Lambda^{s}\tilde{\textbf{v}}^{\varepsilon})_{0}.
\end{equation}

From (\ref{remainder_estimate}), 
$$|Q^{\varepsilon}| \le c(||\tilde{\textbf{v}}^{\varepsilon}||_{s})||\tilde{\textbf{v}}^{\varepsilon}||_{s}$$
and, from the composition theorem in \cite{Metivier},
\begin{equation}
|(\Sigma\Lambda^{s} \textbf{P}J_{\varepsilon} T_{\tilde{G}(J_{\varepsilon}(\tilde{\textbf{v}}^{\varepsilon}+\bar{\textbf{v}}))}, \Lambda^{s} \tilde{\textbf{v}}^{\varepsilon})_{0}| \le c(||\tilde{\textbf{v}}^{\varepsilon}||_{s}) ||\tilde{\textbf{v}}^{\varepsilon}||_{s}.
\end{equation}

It remains to deal with 
$$Re(\Sigma \Lambda^{s} \textbf{P}J_{\varepsilon}T_{i\tilde{A}(\xi, J_{\varepsilon}(\tilde{\textbf{v}}^{\varepsilon}+\bar{\textbf{v}}))}J_{\varepsilon}\tilde{\textbf{v}}^{\varepsilon}, \Lambda^{s} \tilde{\textbf{v}}^{\varepsilon})_{0}=Re(\Sigma J_{\varepsilon} \Lambda^{s} \textbf{P}T_{i\tilde{A}(\xi, J_{\varepsilon}(\tilde{\textbf{v}}^{\varepsilon}+\bar{\textbf{v}}))}J_{\varepsilon}\tilde{\textbf{v}}^{\varepsilon}, \Lambda^{s} \tilde{\textbf{v}}^{\varepsilon})_{0}$$

$$=Re(\Sigma \Lambda^{s}\textbf{P}T_{i\tilde{A}(\xi, J_{\varepsilon}(\tilde{\textbf{v}}^{\varepsilon}+\bar{\textbf{v}}))}J_{\varepsilon}\tilde{\textbf{v}}^{\varepsilon}, \Lambda^{s} J_{\varepsilon}\tilde{\textbf{v}}^{\varepsilon})_{0}+Re([\Sigma, J_{\varepsilon}] \Lambda^{s}\textbf{P}T_{i\tilde{A}(\xi, J_{\varepsilon}(\tilde{\textbf{v}}^{\varepsilon}+\bar{\textbf{v}}))}]J_{\varepsilon}\tilde{\textbf{v}}^{\varepsilon}, \Lambda^{s} \tilde{\textbf{v}}^{\varepsilon})_{0}$$

$$=Re(\Sigma \textbf{P}T_{i\tilde{A}(\xi, J_{\varepsilon}(\tilde{\textbf{v}}^{\varepsilon}+\bar{\textbf{v}}))}\Lambda^{s}J_{\varepsilon}\tilde{\textbf{v}}^{\varepsilon}, \Lambda^{s} J_{\varepsilon}\tilde{\textbf{v}}^{\varepsilon})_{0}+
Re([\Sigma, J_{\varepsilon}]\textbf{P}T_{i\tilde{A}(\xi, J_{\varepsilon}(\tilde{\textbf{v}}^{\varepsilon}+\bar{\textbf{v}}))}\Lambda^{s}J_{\varepsilon}\tilde{\textbf{v}}^{\varepsilon}, \Lambda^{s} \tilde{\textbf{v}}^{\varepsilon})_{0}$$

\begin{equation}
\label{last}
+Re(\Sigma[\Lambda^{s}, \textbf{P}T_{i\tilde{A}(\xi, J_{\varepsilon}(\tilde{\textbf{v}}^{\varepsilon}+\bar{\textbf{v}}))}]J_{\varepsilon}\tilde{\textbf{v}}^{\varepsilon}, \Lambda^{s} J_{\varepsilon}\tilde{\textbf{v}}^{\varepsilon})_{0}.
\end{equation}

In the last term of the expression above, from Proposition 1.10 in \cite{Grenier} the symbol of the commutator $[\Lambda^{s},\textbf{P}T_{i\tilde{A}(\xi, J_{\varepsilon}(\tilde{\textbf{v}}^{\varepsilon}+\bar{\textbf{v}}))}]$ is given by
$$\sum_{|\alpha| \ge 0} \partial^{\alpha}_{\xi}\Lambda^{s} D^{\alpha}_{x}(\sum_{|\beta| \ge 0}\partial^{\beta}_{\xi}\textbf{P}iD^{\beta}_{x}\tilde{A})-D^{\alpha}_{\xi}(\sum_{|\beta| \ge 0}\partial^{\beta}_{\xi}\textbf{P}iD^{\beta}_{x}\tilde{A})D^{\alpha}_{x}\Lambda^{s},$$

where $\sum_{|\beta| \ge 0}\partial^{\beta}_{\xi}(\textbf{P}iD^{\beta}_{x}\tilde{A})D^{\alpha}_{x}$ is the symbol of the composition $\textbf{P}T_{i\tilde{A}}$. Since $\Lambda^{s}(\xi)$ only depends on the parameter $\xi,$ the sum can be written as

$$\Lambda^{s}(\sum_{|\beta| \ge 0}\partial^{\beta}_{\xi}\textbf{P}iD^{\beta}_{x}\tilde{A})-(\sum_{|\beta| \ge 0}\partial^{\beta}_{\xi}\textbf{P}iD^{\beta}_{x}\tilde{A})\Lambda^{s} + \sum_{|\alpha| > 0} \partial^{\alpha}_{\xi}\Lambda^{s} D^{\alpha}_{x}(\sum_{|\beta| \ge 0}\partial^{\beta}_{\xi}\textbf{P}iD^{\beta}_{x}\tilde{A}).$$

Now, since $\Lambda^{s}=(1+|\xi|^{2})^{\frac{s}{2}}Id$, then
$$\Lambda^{s}(\sum_{|\beta| \ge 0}\partial^{\beta}_{\xi}\textbf{P}iD^{\beta}_{x}\tilde{A})-(\sum_{|\beta| \ge 0}\partial^{\beta}_{\xi}\textbf{P}iD^{\beta}_{x}\tilde{A})\Lambda^{s}=0,$$

namely the commutator $[\Lambda^{s}, \textbf{P}T_{i\tilde{A}(J_{\varepsilon}(\tilde{\textbf{v}}^{\varepsilon}+\bar{\textbf{v}}))}]$ has order less than or equal to $s,$ and 
$$|(\Sigma[\Lambda^{s}, \textbf{P}T_{i\tilde{A}(\xi, J_{\varepsilon}(\tilde{\textbf{v}}^{\varepsilon}+\bar{\textbf{v}}))}]\tilde{\textbf{v}}^{\varepsilon}, \Lambda^{s}\tilde{\textbf{v}}^{\varepsilon})_{0}| \le c(||\tilde{\textbf{v}}^{\varepsilon}||_{s})||\tilde{\textbf{v}}^{\varepsilon}||_{s}^{2}.$$

Moreover,
$$[\Sigma, J_{\varepsilon}]=\Sigma J_{\varepsilon}-J_{\varepsilon}\Sigma+\sum_{|\alpha|>0}D^{\alpha}_{\xi}J_{\varepsilon}D^{\alpha}_{x}\Sigma,$$

Since $J_{\varepsilon}=j_{\varepsilon}(\xi) Id$ and $\Sigma$ is symmetric,  $\Sigma J_{\varepsilon}-J_{\varepsilon}\Sigma=0$. Then, the commutator $[\Sigma, J_{\varepsilon}]$ has order less than or equal to $-1$ and, since $\textbf{P} T_{i\tilde{A}(\xi, J_{\varepsilon}(\tilde{\textbf{v}}^{\varepsilon}+\bar{\textbf{v}}))}$ has order $1,$ from the composition theorem in \cite{Metivier} we have
$$([\Sigma, J_{\varepsilon}]\textbf{P}T_{i\tilde{A}(\xi, J_{\varepsilon}(\tilde{\textbf{v}}^{\varepsilon}+\bar{\textbf{v}}))}\Lambda^{s}\tilde{\textbf{v}}^{\varepsilon}, \Lambda^{s}\tilde{\textbf{v}}^{\varepsilon})_{0} \le c(||\tilde{\textbf{v}}^{\varepsilon}||_{s})||\tilde{\textbf{v}}^{\varepsilon}||_{s}^{2}.$$

It remains to consider the last term of (\ref{last}). From (\ref{Lax_symmetrizer}), $\Sigma=(T_{W})^{*}T_{W}+\theta_{\lambda}^{2}(D_{x}),$ and, from Proposition 1.10 in \cite{Grenier}, the only symbol of degree 1 in the expansion of $\Sigma\textbf{P}T_{i\tilde{A}(\xi, J_{\varepsilon}(\tilde{\textbf{v}}^{\varepsilon}+\bar{\textbf{v}}))}$ is given by $$(V^{-1})^{*}V^{-1}\textbf{P}i\tilde{A}(\xi, J_{\varepsilon}(\tilde{\textbf{v}}^{\varepsilon}+\bar{\textbf{v}}))+\theta_{\lambda}^{2}(\xi)\textbf{P}i\tilde{A}(\xi, J_{\varepsilon}(\tilde{\textbf{v}}^{\varepsilon}+\bar{\textbf{v}})).$$ By construction, from (\ref{matrix_V}) and (\ref{inv}), we have

$$\textbf{P}i\tilde{A}(\xi, J_{\varepsilon}(\tilde{\textbf{v}}^{\varepsilon}+\bar{\textbf{v}}))=ViDV^{-1}(\xi, J_{\varepsilon}(\tilde{\textbf{v}}^{\varepsilon}+\bar{\textbf{v}})),$$
where $D(\xi, J_{\varepsilon}(\tilde{\textbf{v}}^{\varepsilon}+\bar{\textbf{v}}))$ is a diagonal matrix of real terms as in (\ref{diagonalization}), and
$$(V^{-1})^{*}V^{-1}\textbf{P}i\tilde{A}(\xi, J_{\varepsilon}(\tilde{\textbf{v}}^{\varepsilon}+\bar{\textbf{v}}))=(V^{-1})^{*}iDV^{-1}(\xi, J_{\varepsilon}(\tilde{\textbf{v}}^{\varepsilon}+\bar{\textbf{v}})).$$

We define
$$N:=(V^{-1})^{*}iDV^{-1}(\xi, J_{\varepsilon}(\tilde{\textbf{v}}^{\varepsilon}+\bar{\textbf{v}})).$$
Then
$$N+N^{*}=0,$$
namely
$$Re(i(V^{-1})^{*}V^{-1}\textbf{P}\tilde{A}(\xi, J_{\varepsilon}(\tilde{\textbf{v}}^{\varepsilon}+\bar{\textbf{v}}))=N+N^{*}=0.$$

The second addend of the symbolic symmetrizer (\ref{Lax_symmetrizer}) gives
\begin{equation}
|Re(i\theta_{\lambda}^{2}(D_{x})\textbf{P}\tilde{A}(\xi, J_{\varepsilon}(\tilde{\textbf{v}}^{\varepsilon}+\bar{\textbf{v}}))\Lambda^{s} \tilde{\textbf{v}}^{\varepsilon}, \Lambda^{s} \tilde{\textbf{v}}^{\varepsilon})_{0}| \le ||\theta_{\lambda}(D_{x})\textbf{P}\tilde{A}(\xi, J_{\varepsilon}(\tilde{\textbf{v}}^{\varepsilon}+\bar{\textbf{v}}))\Lambda^{s} \tilde{\textbf{v}}^{\varepsilon}||_{0} ||\Lambda^{s} \tilde{\textbf{v}}^{\varepsilon}||_{0}.
\end{equation}

From (\ref{theta_lambda}), we get
\begin{equation*}
||\theta_{\lambda}(D_{x})\textbf{P}\tilde{A}(\xi, J_{\varepsilon}(\tilde{\textbf{v}}^{\varepsilon}+\bar{\textbf{v}}))\Lambda^{s} \tilde{\textbf{v}}^{\varepsilon}||_{0} \le \sqrt{1+4 \lambda^{2}} ||\theta_{\lambda}(D_{x})\textbf{P}\tilde{A}(\xi, J_{\varepsilon}(\tilde{\textbf{v}}^{\varepsilon}+\bar{\textbf{v}}))\Lambda^{s} \tilde{\textbf{v}}^{\varepsilon}||_{H^{-1}}
\end{equation*}

\begin{equation}
\le 3 \lambda ||\textbf{P}\tilde{A}(\xi, J_{\varepsilon}(\tilde{\textbf{v}}^{\varepsilon}+\bar{\textbf{v}}))\Lambda^{s} \tilde{\textbf{v}}^{\varepsilon}||_{H^{-1}}
\le c(||\tilde{\textbf{v}}^{\varepsilon}||_{s})||\tilde{\textbf{v}}^{\varepsilon}||_{s}.
\end{equation}

This way, 
$$|Re(\Sigma\textbf{P}T_{i\tilde{A}(J_{\varepsilon}(\tilde{\textbf{v}}^{\varepsilon}+\bar{\textbf{v}}))}\Lambda^{s}J_{\varepsilon}\tilde{\textbf{v}}^{\varepsilon}, \Lambda^{s}J_{\varepsilon} \tilde{\textbf{v}}^{\varepsilon})_{0}| \le c(||\tilde{\textbf{v}}^{\varepsilon}||_{s})||\tilde{\textbf{v}}^{\varepsilon}||_{s}^{2},$$

and, putting it all together, we have

\begin{equation}
\label{Gronwall}
\frac{d}{dt}(\Sigma\Lambda^{s} \tilde{\textbf{v}}^{\varepsilon}, \Lambda^{s} \tilde{\textbf{v}}^{\varepsilon})_{0} \le c(||\tilde{\textbf{v}}^{\varepsilon}||_{s})||\tilde{\textbf{v}}^{\varepsilon}||_{s}^{2}.
\end{equation}

Let $T_{\varepsilon}$ be the maximum time of existence of the solution to system (\ref{Approximation}). We want to show that there exists a time $T>0$, which is independent of $\varepsilon$, such that $T \le T_{\varepsilon}$ for every $\varepsilon>0$. From Theorem \ref{Local Existence of Approximating Solutions_P}, there exists a constant $M$ such that $\displaystyle||\tilde{\textbf{u}}^{\varepsilon}_{0}||_{s} \le M$. Fixed a constant value $\tilde{M}>M,$ let $T_{0}^{\varepsilon} \le T_{\varepsilon}$ be a positive time such that the smooth solution $\tilde{\textbf{v}}^{\varepsilon}$ verifies

\begin{equation}
\label{uniform_estimate}
sup_{~ 0 \le \tau \le T_{0}^{\varepsilon} ~} ||\tilde{\textbf{v}}^{\varepsilon}(\tau)||_{s} \le \tilde{M}.
\end{equation}

From (\ref{Gronwall}), we get
\begin{equation}
\label{prebound1}
||\tilde{\textbf{v}}^{\varepsilon}(t)||_{s} \le ||\tilde{\textbf{v}}^{\varepsilon}_{0}||_{s} e^{c(\tilde{M})t}
\end{equation}

for $t \in [0, T_{0}^{\varepsilon}]$. Let $T,$ with $0 < T \le T_{0}^{\varepsilon},$ be such that

\begin{equation}
M e^{c(\tilde{M})T} \le \tilde{M}.
\end{equation}

This yields

\begin{equation}
\label{time_bound1}
T \le \frac{log(\frac{\tilde{M}}{M})}{c(\tilde{M})}.
\end{equation}

Since $M, \tilde{M}$ are independent of the parameter $\varepsilon$, estimate (\ref{time_bound1}) implies that the time $T$ is independent of $\varepsilon$ and $(\tilde{\textbf{v}}^{\varepsilon})_{\varepsilon \ge 0}$ is uniformly bounded provided that $\displaystyle T \le \frac{log(\frac{\tilde{M}}{M})}{c(\tilde{M})}$.
\end{proof}

\subsection{Uniqueness}
\begin{theorem}\label{unique}
There is a unique  solution $\tilde{\textbf{v}}$ to problem (\ref{BL_compact_translated}) in the space $ Lip([0, T], Lip(\mathbb{R}^{2}) | \nabla \cdot w=0) \cap L^{\infty}([0,T], V^{0})$.
\end{theorem}
\begin{proof}
According to Definition \ref{classical_solutions}, let $\tilde{\textbf{v}}_{1}, \tilde{\textbf{v}}_{2}$ be two solutions to system (\ref{classical_solutions}), with the respective pressure terms $P_{1}, P_{2}$ and the same initial data $\tilde{\textbf{v}}_{1}(0,x)=\tilde{\textbf{v}}_{2}(0,x)=\tilde{\textbf{v}}_{0}$. From (\ref{BL_compact_translated}) and (\ref{P_operator}), we have

\begin{equation*}
\Sigma(\tilde{\textbf{v}}_{2}+\bar{\textbf{v}})\partial_{t}(\tilde{\textbf{v}}_{2}-\tilde{\textbf{v}}_{1})+\sum_{j=1}^{2} \Sigma(\tilde{\textbf{v}}_{2}+\bar{\textbf{v}})\textbf{P}T_{\tilde{A}_{j}(\tilde{\textbf{v}}_{2}+\bar{\textbf{v}})}\partial_{x_{j}}(\tilde{\textbf{v}}_{2}-\tilde{\textbf{v}}_{1})
\end{equation*}

\begin{equation*}
+\Sigma(\tilde{\textbf{v}}_{2}+\bar{\textbf{v}})\sum_{j=1}^{2}[\textbf{P}\tilde{A}_{j}(\tilde{\textbf{v}}_{2}+\bar{\textbf{v}})-\textbf{P}T_{\tilde{A}_{j}(\tilde{\textbf{v}}_{2}+\bar{\textbf{v}})}]\partial_{x_{j}}(\tilde{\textbf{v}}_{2}-\tilde{\textbf{v}}_{1})
\end{equation*}

\begin{equation*}
=\Sigma(\tilde{\textbf{v}}_{2}+\bar{\textbf{v}})\tilde{G}(\tilde{\textbf{v}}_{2}+\bar{\textbf{v}}) - \Sigma(\tilde{\textbf{v}}_{1}+\bar{\textbf{v}})\tilde{G}(\tilde{\textbf{v}}_{1}+\bar{\textbf{v}})
\end{equation*}

\begin{equation*}
+[\Sigma(\tilde{\textbf{v}}_{1}+\bar{\textbf{v}})-\Sigma(\tilde{\textbf{v}}_{2}+\bar{\textbf{v}})]\partial_{t}\tilde{\textbf{v}}_{1}+\sum_{j=1}^{2}[\Sigma(\tilde{\textbf{v}}_{1}+\bar{\textbf{v}})\textbf{P}T_{\tilde{A}_{j}(\tilde{\textbf{v}}_{1}+\bar{\textbf{v}})} - \Sigma(\tilde{\textbf{v}}_{2}+\bar{\textbf{v}})\textbf{P}T_{\tilde{A}_{j}(\tilde{\textbf{v}}_{2}+\bar{\textbf{v}})}] \partial_{x_{j}}\tilde{\textbf{v}}_{1}
\end{equation*}

\begin{equation*}
+ \sum_{j=1}^{2}[\Sigma(\tilde{\textbf{v}}_{1}+\bar{\textbf{v}})\textbf{P}{\tilde{A}_{j}(\tilde{\textbf{v}}_{1}+\bar{\textbf{v}})} - \Sigma(\tilde{\textbf{v}}_{2}+\bar{\textbf{v}})\textbf{P}{\tilde{A}_{j}(\tilde{\textbf{v}}_{2}+\bar{\textbf{v}})}] \partial_{x_{j}}\tilde{\textbf{v}}_{1}
\end{equation*}

\begin{equation}
+\sum_{j=1}^{2}[\Sigma(\tilde{\textbf{v}}_{2}+\bar{\textbf{v}})\textbf{P}T_{\tilde{A}_{j}(\tilde{\textbf{v}}_{2}+\bar{\textbf{v}})}-\Sigma(\tilde{\textbf{v}}_{1}+\bar{\textbf{v}})\textbf{P}T_{\tilde{A}_{j}(\tilde{\textbf{v}}_{1}+\bar{\textbf{v}})}]\partial_{x_{j}}\tilde{\textbf{v}}_{1}.
\end{equation}

As done in Section 3, this provides the following estimate:

\begin{equation}
\label{uniqueness}
\frac{d}{dt}(\Sigma \tilde{\textbf{v}}_{1}-\tilde{\textbf{v}}_{2},  \tilde{\textbf{v}}_{1}-\tilde{\textbf{v}}_{2})_{0} \le c|| \tilde{\textbf{v}}_{1}-\tilde{\textbf{v}}_{2}||_{0}^{2},
\end{equation}
namely $\tilde{\textbf{v}}_{1}=\tilde{\textbf{v}}_{2}=0,$ since $\tilde{\textbf{v}}_{1}-\tilde{\textbf{v}}_{2}(0,x)=\tilde{\textbf{v}}_{1}(0,x)-\tilde{\textbf{v}}_{2}(0,x)=0,$
where the constant value $c$ in (\ref{uniqueness}) only depends on $|\tilde{\textbf{v}}|_{\infty}$, $|\partial_{t}\tilde{\textbf{v}}|_{\infty}$ and $|\nabla \tilde{\textbf{v}}|_{\infty}$. 
\end{proof}

\section{The original biofilms system:  a multi-solid-phases model}
We consider system (\ref{BDEL}), which can be written as

\begin{equation}
\label{BDEL_complete}
\begin{cases}
& \partial_{t}B + \nabla \cdot (B v_{S}) = \Gamma_{B}:=k_{B}BL-k_{D}B, \\
& \partial_{t}D+ \nabla \cdot (D v_{S}) = \Gamma_{D}:=\alpha Bk_{D}-k_{N}D, \\
& \partial_{t}E+ \nabla \cdot (E v_{S}) = \Gamma_{E}:=BLk_{E}-\varepsilon E, \\
& \partial_{t}{v}_{S}+{v}_{S} \cdot \nabla {v}_{S} +\frac{\gamma \nabla (B+D+E)}{ (B+D+E)}+ \nabla {P}=\Gamma_{v_{S}}:=\frac{(M+\Gamma_{B}+\Gamma_{D}+\Gamma_{E})(v_{L}-v_{S})}{B+D+E}, \\
& \partial_{t}{v}_{L}+{v}_{L} \cdot \nabla {v}_{L} + \nabla {P}=\Gamma_{v_{L}}:=\frac{M(v_{S}-v_{L})}{1-(B+D+E)}, \\
& \nabla \cdot ((B+D+E)v_{S} + (1-(B+D+E))v_{L}) = 0, \\
& B+D+E+L=1,\\
\end{cases}
\end{equation}

where $k_{B}, k_{D}, k_{E}, k_{N},\alpha, \varepsilon$ are experimental constants. Now, $\textbf{P}\tilde{A}(\xi, J_{\varepsilon}(\tilde{\textbf{v}}^{\varepsilon}+\bar{\textbf{v}}))$ has the following eigenvalues:

\begin{equation}
\label{BDEL_eigenvalues}
\begin{array}{l}
\lambda_{1}=0,\\ 
\lambda_{2}=(w-Bz)\cdot \xi,\\ 
\lambda_{3}=\lambda_{4}=\lambda_{5}=(w+(1-B)z) \cdot \xi, \\
\lambda_{6/7}=(w+(1-B-\nu))\cdot \xi \pm \sqrt{(1-\nu)(\gamma |\xi|^{2}-\nu(z \cdot \xi)^{2})}. \\
\end{array}
\end{equation}

Besides, the eigenvectors are the columns of $V(\xi, J_{\varepsilon}(\tilde{\textbf{v}}^{\varepsilon}+\bar{\textbf{v}}))$
\begin{equation}
\label{V_BDEL}
\left(\begin{array}{ccccccc}
\frac{q_{11}}{|\xi|\Delta_{1}} & 0 & 1 & -1& 0 & \frac{B|\xi|\Delta_{4}}{\Delta_{2}} & \frac{-B|\xi|\Delta_{4}}{\Delta_{2}} \\
\frac{q_{21}}{|\xi|\Delta_{1}} & 0 & -1 & 0 & 0 & \frac{D|\xi|\Delta_{4}}{\Delta_{2}} & \frac{-D|\xi|\Delta_{4}}{\Delta_{2}} \\
\frac{q_{31}}{|\xi|\Delta_{1}} & 0 & 0 & 1 & 0 &  \frac{E|\xi|\Delta_{4}}{\Delta_{2}} & \frac{-E|\xi|\Delta_{4}}{\Delta_{2}} \\
\frac{q_{41}}{|\xi|\Delta_{1}} & \frac{(1-\nu)\xi_{2}}{|\xi|} & 0 & 0 & \frac{-\nu\xi_{2}}{|\xi|} & \frac{-\xi_{2} \Delta_{5}}{|\xi|\Delta_{2}} & \frac{\xi_{2} \Delta_{5}}{|\xi|\Delta_{2}} \\
\frac{q_{51}}{|\xi|\Delta_{1}} & \frac{-(1-\nu)\xi_{1}}{|\xi|} & 0 & 0 & \frac{\nu\xi_{1}}{|\xi|} & \frac{\xi_{1} \Delta_{5}}{|\xi|\Delta_{2}} & \frac{-\xi_{1} \Delta_{5}}{|\xi|\Delta_{2}} \\
\frac{\xi_{1}}{|\xi|} & \frac{-\xi_{2}}{|\xi|} & 0 & 0 & \frac{-\xi_{2}}{|\xi|} & \frac{\xi_{1}}{|\xi|} & \frac{\xi_{1}}{|\xi|} \\
\frac{\xi_{2}}{|\xi|} & \frac{\xi_{1}}{|\xi|} & 0 & 0 & \frac{\xi_{1}}{|\xi|} & \frac{\xi_{2}}{|\xi|} & \frac{\xi_{2}}{|\xi|} \\
\end{array}\right),
\end{equation}

where $$q_{11}=q_{11}(\xi, J_{\varepsilon}\tilde{\textbf{v}}^{\varepsilon}, \bar{\textbf{v}}), q_{21}=q_{21}(\xi, J_{\varepsilon}\tilde{\textbf{v}}^{\varepsilon}, \bar{\textbf{v}}), q_{31}=q_{31}(\xi, J_{\varepsilon}\tilde{\textbf{v}}^{\varepsilon}, \bar{\textbf{v}}), q_{41}=q_{41}(\xi, J_{\varepsilon}\tilde{\textbf{v}}^{\varepsilon}, \bar{\textbf{v}}),$$
$$q_{51}=q_{51}(\xi, J_{\varepsilon}\tilde{\textbf{v}}^{\varepsilon}, \bar{\textbf{v}})$$ are polynomial functions of degree 3 in the $\xi$ variable, $\Delta_{1}$ in (\ref{Delta}), and
\begin{equation}
\label{Deltas}
\begin{cases}
\nu:=B+D+E, \\
\Delta_{2}:=\gamma |\xi|^{2} - (B+D+E)(z \cdot \xi)^{2}, \\
\Delta_{4}:=\sqrt{(1-\nu)\Delta_{2}},\\
\Delta_{5}:={(\xi_{1}z_{2}-\xi_{2}z_{1})\nu\Delta_{4}},\\
\end{cases}
\end{equation}

while $V^{-1}(\xi, J_{\varepsilon}(\tilde{\textbf{v}}^{\varepsilon}+\bar{\textbf{v}}))$
\begin{equation}
\label{inverse}
\left(\begin{array}{ccccccc}
0 & 0 & 0 & \frac{-\xi_{1}}{|\xi| \sqrt{1-\nu}} & \frac{-\xi_{2}}{|\xi| \sqrt{1-\nu}} & 0 & 0 \\
0 & 0 & 0 & \frac{\xi_{2}}{|\xi|} & \frac{-\xi_{1}}{|\xi|} & \frac{-\xi_{2}\nu}{|\xi|} & \frac{\xi_{1}\nu}{|\xi|} \\
\frac{-D}{\nu} & \frac{B+E}{\nu} & \frac{-D}{\nu} & 0 & 0 & 0 & 0 \\
\frac{-E}{\nu} & \frac{-E}{\nu}  & \frac{B+D}{\nu} & 0 & 0 & 0 & 0 \\
0 & 0 & 0 & \frac{-\xi_{2}}{|\xi|} & \frac{\xi_{1}}{|\xi|} & \frac{-(1-\nu)\xi_{2}}{|\xi|} & \frac{(1-\nu)\xi_{1}}{|\xi|} \\
\frac{\sqrt{\gamma}}{2\nu\sqrt{1-\nu}} & \frac{\sqrt{\gamma}}{2\nu\sqrt{1-\nu}} & \frac{\sqrt{\gamma}}{2\nu\sqrt{1-\nu}} & \frac{\xi_{1}}{2 |\xi| (1-\nu)} & \frac{\xi_{2}}{2|\xi|(1-\nu)} & \frac{\xi_{1}}{2|\xi|} & \frac{\xi_{2}}{2|\xi|} \\
\frac{-\sqrt{\gamma}}{2\nu\sqrt{1-\nu}} & \frac{-\sqrt{\gamma}}{2\nu\sqrt{1-\nu}} & \frac{-\sqrt{\gamma}}{2\nu\sqrt{1-\nu}} & \frac{\xi_{1}}{2 |\xi| (1-\nu)} & \frac{\xi_{2}}{2|\xi|(1-\nu)} & \frac{\xi_{1}}{2|\xi|} & \frac{\xi_{2}}{2|\xi|} \\
\end{array}\right).
\end{equation}

Since $V$ and $V^{-1}$ are bounded for each $\xi \in \mathbb{R}^{2}-\{\textbf{0}\},$ we can apply the arguments developed for system (\ref{BL_System}) to the complete case (\ref{BDEL_complete}).

\section{The three dimensional two phases model}
The three dimensional case contains structural difficulties that we are not able to solve. In three space dimensions
\begin{equation}
\label{PA_symbol_3d}
\textbf{P}{\tilde{A}}(\xi,\textbf{v})=\textbf{P}(\xi)(\tilde{A}_{1}(\textbf{v})\xi_{1}+\tilde{A}_{2}(\textbf{v})\xi_{2}+\tilde{A}_{3}(\textbf{v})\xi_{3}),
\end{equation}

with the following eigenvalues:
\begin{equation}
\label{eigenvalues_3d}
\begin{array}{l}
\lambda_{1}=0, \\
\lambda_{2}=\lambda_{3}=(w-Bz) \cdot \xi, \\
\lambda_{4}=\lambda_{5}=(w+(1-B)z) \cdot \xi, \\
\lambda_{6/7}=(w+(1-2B)z)\cdot \xi \pm \sqrt{(1-B) \Delta_{2}}, \\
\end{array}
\end{equation}

where $\Delta_{2}=\gamma |\xi|^{2}-B(z \cdot \xi)^{2}.$ To simplify the discussion, the matrix with the eigenvectors on the columns, $V(\xi, \textbf{v}),$ has been calculated in the  equilibrium point (\ref{equilibrium_point}) $\bar{\textbf{v}}=(\bar{B}, \bar{w}, \bar{z})=(\bar{B}, \textbf{0}, \textbf{0}),$ with $\bar{B}=1-\frac{k_{D}}{k_{B}}$. We have

$$V(\xi, \textbf{v})=\frac{1}{|\xi|}$$
\begin{equation}
\label{V_matrix_3d}
\left(\begin{array}{ccccccc}
0 & 0 & 0 & 0 & 0 & \bar{B}\sqrt{\frac{1- \bar{B}}{\gamma}}|\xi| & - \bar{B}\sqrt{\frac{1-\bar{B}}{\gamma}}|\xi| \\
-(1-\bar{B})\xi_{1} & \xi_{2}(1-\bar{B}) & \xi_{3}(1-\bar{B}) & -\bar{B}\xi_{2} & -\bar{B}\xi_{3} & 0 & 0 \\
-(1-\bar{B})\xi_{2} & -(1-\bar{B})\xi_{1} & 0 & \bar{B}\xi_{1} & 0 & 0 & 0 \\
-(1-\bar{B})\xi_{3} & 0 & -(1-\bar{B})\xi_{1} & 0 & \bar{B}\xi_{1} & 0 & 0 \\
\xi_{1} & -\xi_{2} & -\xi_{3} & -\xi_{2} & -\xi_{3} & \xi_{1} & \xi_{1} \\
\xi_{2} & \xi_{1} & 0 & \xi_{1} & 0 & \xi_{2} & \xi_{2} \\
\xi_{3} & 0 & \xi_{1} & 0 & \xi_{1} & \xi_{3} & \xi_{3} \\
\end{array}\right).
\end{equation}

For $\xi_{1}=0,$ we get
\begin{equation}
\label{V3d_x10}
\left(\begin{array}{ccccccc}
0 & 0 & 0 & 0 & 0 & \bar{B}\sqrt{\frac{1-\bar{B}}{\gamma}}|\xi| & -\bar{B}\sqrt{\frac{1-\bar{B}}{\gamma}}|\xi| \\
0 & (1-\bar{B})\xi_{2} & (1-\bar{B})\xi_{3} & -\bar{B}\xi_{2} & -\bar{B}\xi_{3} & 0 & 0 \\
-(1-\bar{B})\xi_{2} & 0 & 0 & 0 & 0 & 0 & 0 \\
-(1-\bar{B})\xi_{3} & 0 & 0 & 0 & 0 & 0 & 0 \\
0 & -\xi_{2} & -\xi_{3} & -\xi_{2} & -\xi_{3} & 0 & 0 \\
\xi_{2} & 0 & 0 & 0 & 0 & \xi_{2} & \xi_{2} \\
\xi_{3} & 0 & 0 & 0 & 0 & \xi_{3} & \xi_{3} \\
\end{array}\right).
\end{equation}

The second and the third columns of (\ref{V3d_x10}), namely the second and the third coincident eigenvalues in (\ref{eigenvalues_3d}), degenerate in the same vector when $\xi_{1}=0.$ This happens also to the fourth and the fifth columns of (\ref{V3d_x10}), i.e. the fourth and the fifth coincident eigenvalues in (\ref{eigenvalues_3d}). For this reason, in the three dimensional case the symbol $\textbf{P}\tilde{A}(\xi, \textbf{v})$ in (\ref{PA_symbol_3d}) loses the property of \emph{strong symmetrizability} and related \emph{microlocal symmetrizability}, according to the definitions given in \cite{Metivier}.


\begin{thebibliography}{11}

\bibitem{astanin}
\textsc{Astanin S. \& Preziosi L.} (2008)
\newblock Multiphase models of tumor growth., 
Selected topics in cancer modeling,
\newblock {\em Model. Simul. Sci. Eng. Technol., }
Birkh\"auser  Boston, Boston, MA, 
223--253.

\bibitem{Da Veiga}
\textsc{Beir\~{a}o da Veiga H. \& Valli A.} (1980)
\newblock Existence of $C^{\infty}$ solutions of the Euler Equations for non-homogeneous fluids,
\newblock  {\em Comm. Part. Diff. Eq.}
5,
95-107.

\bibitem{Benzoni}
\textsc{Benzoni-Gavage S. \& Serre D.} (2007)
\newblock Multidimensional Hyperbolic Partial Differential Equations,
\newblock{\em Oxford University Press.}

\bibitem{Bertozzi}
\textsc{Bertozzi A. \& Majda A.} (2002)
\newblock Vorticity and Incompressible Flow,
\newblock{\em Cambridge University Press.}

\bibitem{Bianchini1}
\textsc{Bianchini R. \& Natalini R.} (2016)
\newblock Global existence and asymptotic stability of smooth solutions to a fluid dynamics model of biofilms in one space dimension,
\newblock {\em J. Math. Anal. Appl.,}
{434},
1909-1923.

\bibitem{Bianchini}
\textsc{Bianchini R. \& Natalini R.} (2016)
\newblock Well-posedness of a model of nonhomogeneous compressible-incompressible fluids,
\newblock {\em submitted preprint.}

\bibitem{Bowen1}
\textsc{Bowen R. M.} (1976)
\newblock{Theory of mixtures},
\newblock {\em Continuum physics, Vol.3, Eringen A.C. ed., Academic Press}.

\bibitem{Bowen2}
\textsc{Bowen R. M.} (1980)
\newblock{Incompressible porous media model by use of theory of mixtures,}
\newblock {\em Int. J. Eng. Sci.,}
{18},
1129-1148.

\bibitem{cdnr}
\textsc{Clarelli F., Di Russo C., Natalini R. \& Ribot M.} (2013)
\newblock A fluid dynamics model of the growth of phototrophic biofilms,
\newblock  {\em J. Math. Biol.,}
{66}(7),
1387--1408.

\bibitem{Danchin}
\textsc{Danchin R.} (2010)
\newblock On the well-posedness of the incompressible density-dependent Euler equations in the $L^{p}$ framework,
\newblock  {\em J. Diff. Eq.}
{24}(8),
2130-2170.

\bibitem{Ehlers}
\textsc{Ehlers W.} (1993)
\newblock Constitutive equations for granular materials in geomechanical context,
\newblock {\em Environmental Sciences and Geophysics, CISM Courses and Lectures N.337, Hutter K. ed., Springer-Verlag}.

\bibitem{Farina}
\textsc{Farina A. \& Preziosi L.} (2002)
\newblock On Darcy's law for growing porous media,
\newblock  {\em In. J. Non-Lin. Mech.}
37(3),
485-491.

\bibitem{Grenier}
\textsc{Grenier E.} (1997)
\newblock Pseudo-Differential Energy Estimates of Singular Perturbations,
\newblock  {\em Comm. Pure Appl. Math.}
{50}(9),
821-865.

\bibitem{Gud1}
\textsc{Gudmundsson R.L.} (2002)
\newblock On the well-posedness of the two-fluid model for dispersed two-phase flow in 2D,
\newblock {\em TechnicalReport TRITA-NA-0223, RoyalInstitute of Technology}

\bibitem{Ishii}
\textsc{Ishii M. \& Song J.H.} (2000)
\newblock The well-posedness of incompressible one dimensional two-fluid model, 
\newblock {\em International Journal of Heat and Mass Transfer}
{43}, 
2221–2231.

\bibitem{Parabello}
\textsc{Hanich L., Louaked M., Thompson C. P.} (2003)
\newblock Well-posedness of incompressible models of two - and three - phase flow,
\newblock {\em IMA Journal of Applied Mathematics}
{68},
595-620.
 
\bibitem{Klainerman}
\textsc{Klainerman S. \& Majda A.} (1981)
\newblock Singular Limits of Quasilinear Hyperbolic Systems with Large Parameters and the Incompressible Limit of Compressible Fluids,
\newblock {\em Comm. Pure Appl. Math.}
{XXXIV},
481-524.

\bibitem{Lad}
\textsc{Ladyzhenskaya A. O. \& Solonnikov V. A.} (1978)
\newblock Unique solvability of an initial and boundary value problem for viscous incompressible non-homogeneous fluids,
\newblock {\em Clarendon Press, Oxford}.
9,
697–749.

\bibitem{Lions}
\textsc{Lions P. L.} (1996)
\newblock Mathematical Topics in Fluid Mechanics,
\newblock {\em J. Soviet. Math.}

\bibitem{Majda}
\textsc{Majda A.} (1984)
\newblock Compressible Fluid Flow and Systems of Conservation Laws in Several
Space Variables, 
\newblock {\em Springer-Verlag, New York}.

\bibitem{Marsden}
\textsc{Marsden J. E.} (1976)
\newblock Well-posedness of the equations of a non-homogeneous perfect fluid,
\newblock {\em Comm. Part. Diff. Eq.}
1,
215-230.

\bibitem{Metivier}
\textsc{M\'etivier G.} (2008)
\newblock Para-differential Calculus and Application to the Cauchy Problem for Nonlinear Systems,
\newblock {\em CRM Series, Edizioni della Scuola Normale Superiore}.

\bibitem{Muller1}
\textsc{Muller I.} (1975)
\newblock Thermodynamics of mixture of fluids,
\newblock {\em J. Mec.,}
{14},
267-303.

\bibitem{Muller2}
\textsc{Muller I.} (1985)
\newblock Rational Thermodynamics of mixtures of fluids,
\newblock {\em Thermodynamics and Constitutive Equations, Lecture Notes in Physics 228, Grioli G. ed., Springer-Verlag.}

\bibitem{Rajagopal} 
\textsc{Rajagopal K.R. \& Tao L.} (1995)
\newblock Mechanics of Mixtures,
\newblock Series on Advances in Mathematics for Applied Sciences, 35. World Scientific Publishing Co., River Edge, NJ.       

\bibitem{Schochet}
\textsc{Schochet S.} (1986)
\newblock The Compressible Euler Equations in a Bounded Domain: Existence of Solutions and the Incompressible Limit,
\newblock {\em Communications in Mathematical Physics}
{104},
49-75.

\bibitem{Stewart}
\textsc{Stewart H. B. \&  Wendroff. B.} (1984)
\newblock Review article two-phase flow: Models and methods. 
\newblock{Journal of Computational Physics} 
{56},
363–409. 

\bibitem{Taylor}
\textsc{Taylor M.} (1996)
\newblock Partial differential equations III,
\newblock {\em Applied Mathematical Sciences 117, Springer}.  

\bibitem{Temam}
\textsc{Temam R.} (1977)
\newblock Navier-Stokes Equations -Theory and Numerical Analysis,
\newblock {\em North-Holland Publishing Company}.  

\bibitem{Valli}
\textsc{Valli A. \& Zajazckowski W. M.} (1988)
\newblock About the motion of nonhomogeneous ideal incompressible fluids,
\newblock {\em Nonlinear Analysis, Theory, Methods} \& {\em Applications}
{Vol.12, No. 1}
{43-50}.

\bibitem{Ystrom}
\textsc{ Ystr\"om} (2001)
\newblock On two-fluid equations for dispersed incompressible two-phase flow,
\newblock {\em Comput. Visual. Sci.} 
{4},
125–135. 

\end{thebibliography}
\end{document}